\documentclass[10pt,a4paper,reqno]{article}
\usepackage{graphicx}
\usepackage{amsmath,amsthm,amsfonts,amssymb, mathabx}
\usepackage[abbrev,non-sorted-cites]{amsrefs}
\usepackage{latexsym}
\usepackage{color}
\usepackage{authblk}

\DeclareMathOperator{\im}{Im}
\DeclareMathOperator{\re}{Re}

\DeclareMathOperator{\supp}{supp}

\newtheorem{theorem}{Theorem}[section]
\theoremstyle{plain}
\newtheorem*{theorem*}{Main Theorem}

\newtheorem{corollary}[theorem]{Corollary}

\newtheorem{definition}[theorem]{Definition}

\newtheorem{lemma}[theorem]{Lemma}

\newtheorem{proposition}[theorem]{Proposition}
\newtheorem{remark}{Remark}

\begin{document}

\title{On the Gluing Construction of Translating Solitons to Lagrangian Mean Curvature Flow}
\author{WEI-BO SU}
\affil{\it Academia Sinica, Taipei, Taiwan\\
wbsu@gate.sinica.edu.tw}
\date{}
\maketitle
\begin{abstract}
We construct Lagrangian translating solitons by desingularizing the intersection points between Lagrangian Grim Reaper cylinders with the same phase using special Lagrangian Lawlor necks. The resulting Lagrangian translating solitons could have arbitrarily many ends and non-contractible loops.
\end{abstract}

\tableofcontents

\section{Introduction}
Let $(X^{2m}, J, \omega, \Omega)$ be a Calabi--Yau $m$-fold. It is conjectured that one should be able to define the derived Fukaya category $D^{b}\mathcal{F}(X)$ of Lagrangian submanifolds in $X$ and there should exist a Bridgeland stability condition on $D^{b}\mathcal{F}(X)$ so that the stable objects are represented by {\it special Lagrangians}, namely, those real $m$-submanifolds where $\omega$ and $\im\Omega$ restrict to zero. Special Lagrangians are calibrated submanifolds due to Harvey--Lawson \cite{HL}, and hence they are volume-minimizing. Joyce \cite{JoyceConj} conjectured that such stability condition could be found by {\it Lagrangian mean curvature flow (LMCF)}, that is, an evolution of Lagrangian submanifolds $F_{t}:L\to X$ obeying the rule
\begin{align*}
    \frac{d}{dt}F_{t} = H_{F_{t}},
\end{align*}
where $H_{F_{t}}$ is the mean curvature vector of $F_{t}(L)$ in $X$. LMCF is the negative gradient flow of the volume functional, so it is hoped that LMCF will deform a Lagrangian submanifold (possibly with finite times of ``surgeries'') to a union of special Lagrangians which realizes the Harder-Narasimhan filtration of the stability condition. It is therefore important to study LMCF and its singularity formation. When the initial Lagrangian submanifold $F:L\to X$ is {\it zero-Maslov}, that is, $F^{*}\Omega = e^{i\theta_{L}}dV_{L}$ for some real-valued phase function $\theta_{L}:L\to\mathbb{R}$, Neves \cite{NevesSing} showed that LMCF starting from $F:L\to X$ will not develop Type I (fast forming) singularities, so in this case one should focus on  Type II (slowly forming) singularities and by a blow-up procedure, such kind of singularities are modeled on {\it eternal solutions} in $\mathbb{C}^{m}$, namely, the solutions to LMCF in $\mathbb{C}^{m}$ which are defined for all $t\in\mathbb{R}$. 

One major class of eternal solutions to LMCF in $\mathbb{C}^{m}$ is the translating solutions, that is, solutions to LMCF of the form $F_{t} = F - tT$, where $F$ is a Lagrangian immersion and  $T\in\mathbb{C}^{m}$ is a fixed vector. The Lagrangian immersion $F:L\to\mathbb{C}^{m}$ should satisfy the equation
\begin{align}\label{translatoreq}
    H_{F} + T^{\perp} = 0,
\end{align}
and will be called the {\it translating soliton} for LMCF. Perhaps the simplest nontrivial example of  Lagrangian translating soliton is the {\it Grim Reaper curve} in $\mathbb{C}$, which can be parametrized by
\begin{align*}
\gamma_{0}:(-\pi/2, \pi/2)\to\mathbb{C},\quad\gamma_{0}(x)=-\log\cos(x) - i(x-\pi/2).
\end{align*}
We can then construct $m$-dimesional Lagrangian translating soliton by Cartesian product the Grim Reaper curve with $\mathbb{R}^{m-1}$, which we shall call it a {\it Grim Reaper cylinder}. More precisely, we have a Lagrangian translating soliton $L_{0}$ in $\mathbb{C}^{m}$ defined by
\begin{align}\label{basicGR}
    L_{0} := \{(x_{1}, \cdots, x_{m-1}, \gamma_{0}(x_{m}))\:|\:(x_{1}, \cdots, x_{m-1}, x_{m})\in\mathbb{R}^{m-1}\times(-\pi/2, \pi/2)\}.
\end{align}
It has been shown that Grim Reaper cylinders arise in the singularity formation of LMCF \cite{SmoWhitney}. More complicated examples of Lagrangian translating solitons are constructed by Castro--Lerma \cite{CL} in $\mathbb{C}^{2}$, and  Joyce--Lee--Tsui \cite{JLT} in $\mathbb{C}^{m}$ for $m\geq 2$. In particular, the $2$-dimensional examples constructed by Castro--Lerma are generated by two plane curves that satisfy the curve shortening flow, and the example constructed by Joyce--Lee--Tsui are {\it almost-calibrated}, namely, the total variation of the phase function $\theta_{L}$ is less than $\pi$. It is still unknown whether these examples can arise as blow-up limits of finite-time singularities of LMCF.

In the present paper, we aim to construct new examples of Lagrangian translating solitons by {\it gluing construction}. More precisely, we first construct Lagrangian translating solitons {\it with isolated conical singularities} by intersecting suitably rotated Grim Reaper cylinders (see Figure 4). The tangent cones of the intersection points would satisfy an angle condition so that there is a special Lagrangian {\it ``Lawlor neck''}, unique up to scaling, asymptotic to each tangent cone. One then removes a small neighborhood around each intersection point and replace it with a suitably scaled Lawlor neck. The resulting submanifold, the {\it ``approximate solution''}, is Lagrangian, but does not satisfy $(\ref{translatoreq})$, so it is not a translating soliton. However, one can show that the approximate solution is close to being a translating soliton in a precise sense provided the Lawlor necks are {\it small}, so we can perturb such an approximate solution to yield a Lagrangian translating soliton. 

We now describe our building blocks in more details in order to state our main result. Given $\mathbf{\phi} = (\phi_{1}, \cdots, \phi_{m})\in(0, \pi)^{m}$ and $\lambda\in\mathbb{R}$, define an automorphism
\begin{align}\label{AffineTrans}
    &P_{(\mathbf{\phi},  \lambda)}:\mathbb{C}^{m}\to\mathbb{C}^{m}, \;\mbox{ where}\nonumber\\
    &P_{(\mathbf{\phi}, \lambda)}(z_{1}, \cdots, z_{m}) := (e^{i\phi_{1}}z_{1}, \cdots, e^{i\phi_{m-1}}z_{m-1}, z_{m} + (\lambda + i\phi_{m})).
\end{align}
Define the set of admissible angles
\begin{align}\label{AdmissAngles}
    \mathcal{A} := \{\mathbf{\phi} = (\phi_{1}, \cdots, \phi_{m})\in (0, \pi)^{m}\:|\:\sum_{j=1}^{m}\phi_{j} = \pi\}.
\end{align}
Let $L_{0}$ be the Lagrangian translating soliton as defined in $(\ref{basicGR})$. If $(\mathbf{\phi}, \lambda)\in (0, \pi)^{m}\times\mathbb{R}$ such that $\mathbf{\phi}\in\mathcal{A}$, then there exists a unique intersection point $p\in\mathbb{C}^{m}$ of $L_{0}$ and $P_{(\mathbf{\phi}, \lambda)}(L_{0})$ such that the tangent cone $T_{p}L_{0}\cup T_{p}P_{(\mathbf{\phi}, \lambda)}(L_{0})$ satisfies the {\it angle criterion}, namely, up to a $U(m)$-rotation, $T_{p}L_{0}\cup T_{p}P_{(\mathbf{\phi}, \lambda)}(L_{0})$ is identified with the pair of planes $\Pi_{0}\cup\Pi_{\phi}$, where
\begin{align*}
    \Pi_{0} = \{(x_{1}, \cdots, x_{m})\:|\:x_{j}\in\mathbb{R}\}\quad\Pi_{\phi} = \{(e^{i\phi_{1}}x_{1}, \cdots, e^{i\phi_{m}}x_{m})\:|\:x_{j}\in\mathbb{R}\}
\end{align*}
with $\sum_{j=1}^{m}\phi_{j} = \pi$. The Lawlor neck $N:S^{m-1}\times\mathbb{R}\to\mathbb{C}^{m}$ is the unique (up to scaling) smooth, exact, embedded special Lagrangian submanifold asymptotic to $\Pi_{0}\cup\Pi_{\phi}$. For $t>0$, the family $tN$ converges weakly to $\Pi_{0}\cup\Pi_{\phi}$ as $t\to 0$. Hence, by taking $t$ sufficiently small, we can make $tN$ as close to the tangent cone at $p$ as we want. 
The main result of this paper may be stated as follows.
\begin{theorem*}
Given $(\mathbf{\phi}, \lambda)$ with $\mathbf{\phi}\in\mathcal{A}$. There exists $\delta>0$ a $1$-parameter family of smooth, embedded Lagrangian translating solitons $L_{t}$, $t\in (0, \delta)$, diffeomorphic to $S^{m-1}\times\mathbb{R}$, such that $L_{t}$ converges weakly to $L_{0}\cup P_{(\mathbf{\phi}, \lambda)}(L_{0})$ as $t\to 0$.
\end{theorem*}

\begin{center}
\includegraphics[width=13cm]{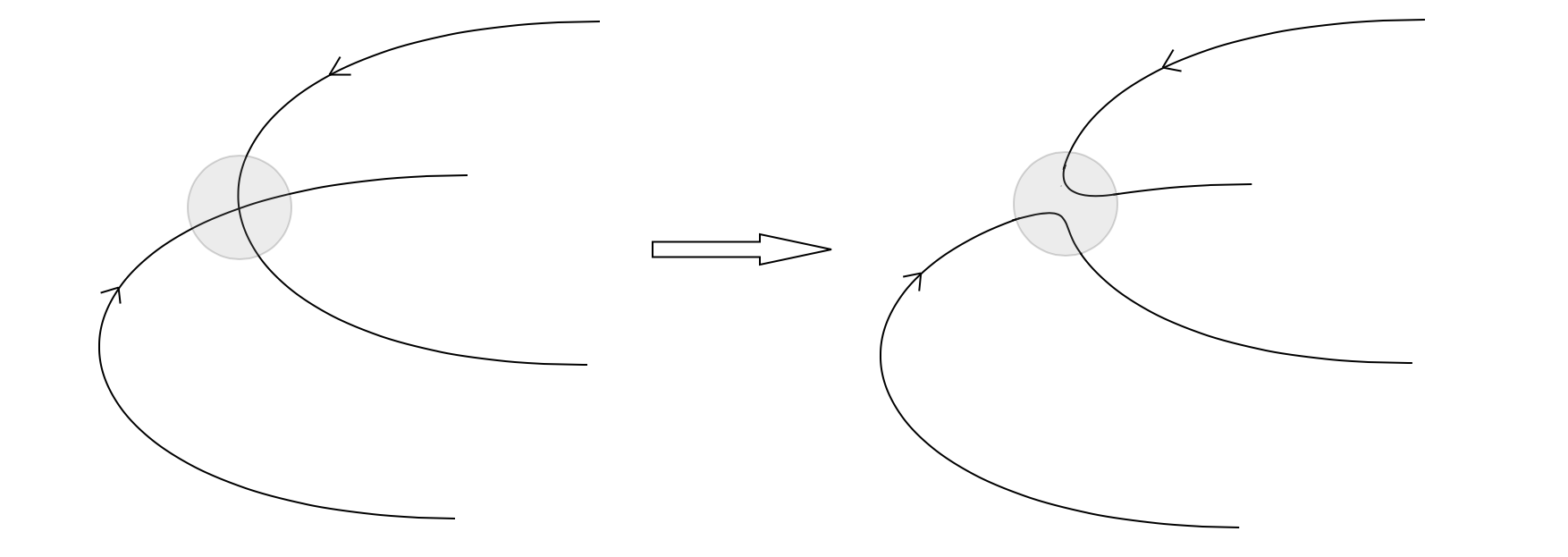}\\
\small{Figure 1. Schematic picture of resolution of an isolated singular point.}
\end{center}

By composing the maps $P_{(\mathbf{\phi}, \lambda)}$ with different $(\phi, \lambda)$ with $\phi\in\mathcal{A}$ finiely many times, we may construct more complicated Lagrangian translating solitons with isolated conical singulairities such that the tangent cone at each singularity is a pair of planes satisfies the angle criterion. By the same gluing technique, we may construct Lagrangian translating solitons with complicated topologies. In particular,
\begin{corollary}
There exist smooth Lagrangian translating solitons in $\mathbb{C}^{m}$, $m\geq 3$, constructed by desingularizing intersection points of Grim Reaper cylinders, with arbitrarily many ends and non-contractible loops.
\end{corollary}
 
The approximate solutions can also be constructed by gluing Lawlor necks to the intersection points of Grim Reaper cylinders and suitably rotated Lagrangian planes (see Figure 4), and that is indeed the author's original plan. However, the author fails to do the perturbation for this kind of approximate solutions since our weighted Sobolev spaces are not able to control the functions in the $T$-direction (see page 14). Because of this, we can only do the perturbation for the so-called {\it $T$-finite} approximate solution (see Definition $\ref{Tfinite}$), namely, those approximate solutions which do not extend to the infinity in $T$-direction.

The reason why the above gluing procedure can be applied to construct Lagrangian translating solitons is the following. In \cite{WBfmin}, we study the {\it $f$-special Lagrangians} in gradient steady K\"ahler--Ricci solitons. To be precise, let $(X, \omega, f)$ be a gradient steady K\"ahler--Ricci soliton, then a result of R. Bryant shows that there exists a holomorphic $(m, 0)$-form $\Omega_{f}$ such that $(X, \omega, \Omega_{f})$ is an {\it almost Calabi--Yau} structure satisfying
\begin{align*}
    e^{-f}\frac{\omega^{m}}{m!} = (-1)^{\frac{m(m-1)}{2}}(i/2)^{m}\Omega_{f}\wedge\overline{\Omega_{f}},
\end{align*}
and $f$-special Lagrangians with phase $\overline{\theta}$ are the $m$-submanifolds defined by the condition
\begin{align*}
    \omega\big|_{L} = \im(e^{-i\overline{\theta}}\Omega_{f})\big|_{L} = 0.
\end{align*}
Note that if $f$ is a constant, $(X, \omega, f)$ reduces to a Calabi--Yau structure. If $F:L\to X$ is $f$-special Lagrangian with an isolated singularity at $p\in X$, then in a small neighborhood $U\ni p$, the value of $f\big|_{U}$ is very close to $f(p)$, hence $f$ is almost a constant. Therefore we may use special Lagrangians in Calabi--Yaus to resolve the singularities of $f$-special Lagrangians, as long as the surgery region is small enough. Now a key observation is that, if we take $(X, \omega, f) = (\mathbb{C}^{m}, \omega_{0}, f(z) = 2\langle z, T\rangle)$, then one can show that $f$-special Lagrangians in this case are exactly the Lagrangian translating solitons satisfying $(\ref{translatoreq})$. Hence we may resolve the singularities of Lagrangian translating solitons by special Lagrangians in $\mathbb{C}^{m}$. Another viewpoint is that since the mean curvature vector $H_{F}$ scales like $H_{\lambda F} = H_{F}/\lambda$ for $\lambda>0$, if $F$ is a Lagrangian translating soliton, then
\begin{align*}
    H_{\lambda F} = \frac{1}{\lambda}H_{F} = -\frac{1}{\lambda}T^{\perp}.
\end{align*}
So we see that $\lambda F$ becomes closer and closer to a minimal submanifold as $\lambda\to\infty$. As a result, at a large scale, the Lagrangian translating solitons approximate special Lagrangians in $\mathbb{C}^{m}$.

The key issue when doing the perturbation argument is the {\it uniform invertibility} of the linearized operator. To describe the idea, let $\widetilde{L}_{t}$, $t\in(0, \delta)$ be the family of approximate solutions obtained by gluing in scaled Lawlor necks $tN$. Then we linearize the translator equation $(\ref{translatoreq})$ at each $\widetilde{L}_{t}$ to get a family of linear operators $\mathcal{L}_{t}$. As a family of operators between ordinary function spaces like $C^{k, \alpha}$ or $W^{k, p}$, $\mathcal{L}_{t}$ is in general not Fredholm, or even if $\mathcal{L}_{t}$ is invertible for each $t$, the operator norm of $\mathcal{L}^{-1}_{t}$ is dependent on $t$ and often blow up as $t\to 0$, due to the lack of uniform curvature bound on $\widetilde{L}_{t}$. These difficulties could possibly get round by considering {\it weighted} function spaces. Roughly speaking, one uses the weight to compensate for the possible blow-up behavior of the functions causing by either noncompactness of the manifolds or blow-up of the curvature. In our case, we define the weighted Sobolev spaces $W^{k, p}_{\beta, \gamma, t}$ on $\widetilde{L}_{t}$ with the norm 
\begin{align*}
\|u\|_{W^{k, p}_{\beta, \gamma, t}} := \left(\sum_{j=0}^{k}\int_{\widetilde{L}_{t}}|e^{\frac{\beta f_{t}}{2}}\rho^{-\gamma+j}_{t}\nabla^{j}u|^{p}\:\rho^{-m}_{t}\:dV_{g_{t}}\right)^{\frac{1}{p}}.
\end{align*}
Notice that these weighted Sobolev spaces are in fact a kind of ``interpolation'' between two different weighted Sobolev spaces of functions supported on different parts of the apporximate solutions, one is $W^{k, p}_{\beta}$ with weight $e^{\beta f_{t}}$ defined on the Grim Reaper ends, and another one is the usual weighted Sobolev spaces $W^{k, p}_{\gamma, t}$ with weight $\rho_{t}$ defined on asymptotically conical Riemannian manifolds. Hence, roughly speaking, the weight function $e^{\beta f_{t}}$ governs the behavior of the functions near spatial infinity, and the weight function $\rho_{t}$ controls the functions on the scaled necks $tN$. Using these weighted Sobolev spaces, we are able to obtain uniform invertibility result for $\mathcal{L}_{t}$ (Proposition $\ref{LinearIsom}$) by first show the uniform invertibility of $\mathcal{L}_{t}$ for functions supported on either Grim Reaper ends and necks using the weighted spaces $W^{k, p}_{\beta}$ and $W^{k, p}_{\gamma, t}$, respectively, and then by a cut-off function argument by Pacini {\cite{PaciniUnifEst}} to achieve uniform invertibility on the whole manifold.

The gluing construction has been proven to be useful for constructing geometric objects satisfying some nonlinear elliptic PDEs. The most relevant results to this paper are the gluing construction of compact special Lagrangians in almost Calabi--Yau manifolds by Joyce \cite{Joyce3}, and the gluing construction of special Lagrangian conifolds in $\mathbb{C}^{m}$ by Pacini \cite{PaciniGlue}. Besides the results in the Lagrangian category, there are gluing constructions for $2$-dimensional translating solitons in $\mathbb{R}^{3}$, by desingularizing the intersection curve of some known examples of translating solitons by the Scherk surface, for instance, see \cites{Trident, Transsurface2013, Transsurface2015, Transsurface2017}. There are also gluing constructions for steady K\"ahler--Ricci solitons \cites{conlon2020steady, biquard2017steady}.

This paper is organized as follows. In section 2, we review the geometry of Lagrangian translating solitons. In section 3, we introduce the building blocks for our gluing construction and construct examples of Lagrangian translating solitons with isolated conical singularities by intersecting Grim Reaper cylinders. In section 4, we define the approximate solutions and estimate the error term. Section 5 and 6 are devoted to the perturbation argument. Then we prove the uniform invertibility of the linearized operator in section 5. In section 6, we estimate the quadratic term and solve the perturbation problem using Fix Point Theorem. 

We end this section with some remarks on the abuse of notations. Let $F: L\to X$ be an immersed submanifold. We will sometimes omit the immersion $F$ but just saying that $L$ is a submanifold of $X$. We will also not distinguish different constants in a sequence of estimates if they do not depend on the relevant quantities.

\subsection*{Acknowledgement}
The author would like to thank Dominic Joyce for suggesting him to think about gluing construction for Lagrangian translating solitons. He also wants to thank Yng--Ing Lee for stimulating discussions.

\section{$f$-special Lagrangian Geometry}
\subsection{Preliminaries}

We briefly review the relevant theory of Lagrangian submanifolds in almost Calabi--Yau manifolds and fix notations.

Let $X$ be a real $2m$-dimensional manifold, $m\geq 1$. Recall that a quadruple $(X, J, \omega, \Omega_{f})$ is called an {\it almost Calabi--Yau $m$-fold} if $(X, J, \omega)$ is a K\"ahler manifold and $\Omega_{f}$ is a nonvanishing holomorphic $(m, 0)$-form satisfying
\begin{align}\label{ACY}
    e^{-f}\:\frac{\omega^{m}}{m!} = (-1)^{\frac{m(m-1)}{2}}\left(\frac{i}{2}\right)^{m}\Omega\wedge\overline{\Omega}
\end{align}
for some smooth function $f:X\to\mathbb{R}$. If $f \equiv 0$, then $(X, J, \omega, \Omega_{0})$ is called a {\it Calabi--Yau $m$-fold}. In this case, we denote $\Omega := \Omega_{0}$. 

A real $m$-dimensional immersed submanifold $F:L\to X$ is a {\it Lagrangian } submanifold if 
\begin{align}\label{Lag}
    F^{*}\omega = 0.
\end{align}
From this condition and $(\ref{ACY})$, there is a (multi-valued) function $\theta_{L}: L\to\mathbb{R}/2\pi\mathbb{Z}$ so that
\begin{align}\label{fcalibrated}
    F^{*}\Omega = e^{i\theta_{L}}\:e^{-\frac{F^{*}f}{2}}dV_{g},
\end{align}
where $dV_{g}$ is the volume form of the induced metric $g$ on $L$. $\theta_{L}$ is called the {\it Lagrangian angle} of $L$. If $\theta_{L}$ can be lifted to a single-valued function $\theta_{L}:L\to\mathbb{R}$, then the Lagrangian submanifold $L$ is said to be of {\it zero Maslov class}. In any case, $d\theta_{L}$ is a well-defined $1$-form on $L$. An important fact is that this $1$-form  $d\theta_{L}$ is dual to the {\it generalised} mean curvature of $L$, which we are going to explain. Let $H_{L}$ be the mean curvature vector of the Lagrangian immersion $F:L\to X$, then the {\it generalised mean curvature vector} $K_{L}$ is defined by $K_{L} := H_{L} + (\frac{1}{2}\overline{\nabla}f)^{\perp}$, where $\overline{\nabla}$ denotes the gradient of the ambient K\"ahler metric $\overline{g} := \omega(\cdot, J\cdot)$ and $\perp$ denotes the orthogonal projection to the normal bundle of $L$. Then we have the identity
\begin{align}
    d\theta_{L} = F^{*}\iota_{K_{L}}\omega,
\end{align}
or equivalently, $K_{L} = JF_{*}(\nabla\theta_{L})$, where $\nabla$ here denotes the tangential part of $\overline{\nabla}$ on $L$. Therefore, Lagrangian submanifolds with constant Lagrangian angle $\theta_{L} = \overline{\theta}$ have vanishing generalised mean curvature. In this case, we say $F:L\to X$ is an {\it $f$-special Lagrangian submanifold with phase $\overline{\theta}$}. 

From $(\ref{fcalibrated})$, if $F:L\to X$ is special Lagrangian with phase $\overline{\theta}$, then $F^{*}\re(e^{-i\overline{\theta}}\Omega) = e^{-\frac{F^{*}f}{2}}dV_{L}$. In other words, $F:L\to X$ is {\it f-calibrated} by $\re (e^{-i\overline{\theta}}\Omega_{f})$ in the sense of \cite{WBfmin}. It follows that any compact subset $K\subset L$ of an  $f$-special Lagrangian submanifold $F:L\to X$ minimizes the $f$-volume $V_{f}(K):=\int_{K}e^{-\frac{f}{2}}dV_{L}$ homologically. If $f\equiv 0$, that is, if $(X, J, \omega, \Omega)$ is Calabi--Yau, then $f$-special Lagrangians reduce to the usual {\it special Lagrangians}.

\begin{remark}
In Joyce's terminology (see \cite[Definition   3.5]{JoyceCS1}), $f$-special Lagrangians are simply called special Lagrangians. However, we think it would be clearer to the readers if we emphasize the ``weight'' $f$. 
\end{remark}

\subsection{Lagrangian translating solitons}
We now consider $\mathbb{C}^{m}$ equipped with standard Euclidean metric $g_{0}$, complex structure $J$, K\"ahler form $\omega_{0}$, and holomorphic $(m, 0)$-form $\Omega = dz_{1}\wedge\cdots\wedge dz_{m}$. Then $(\mathbb{C}^{m}, J, \omega_{0}, \Omega)$ is Calabi--Yau.
\begin{definition}
An immersed Lagrangian submanifold $F:L\to\mathbb{C}^{m}$ is a Lagrangian translating soliton if 
\begin{align*}
H_{F} + T^{\perp} = 0
\end{align*}
for some fixed vector $T\in\mathbb{C}^{m}$.
\end{definition}
Let $\theta_{L}:L\to\mathbb{R}/2\pi\mathbb{Z}$ be the Lagrangian angle of $F$ with respect to $\Omega$, then by Lagrangian condition,
\begin{align*}
\begin{array}{rcl}
H_{F} + T^{\perp} = 0 &\Longleftrightarrow & J\nabla\theta_{L} + T^{\perp} = 0\\
&\Longleftrightarrow & \nabla\theta_{L} - (JT)^{T} = 0\\
&\Longleftrightarrow & \nabla\theta_{L} - \nabla\langle F, JT\rangle = 0.
\end{array}
\end{align*}
So $F$ is a Lagrangian translating soliton if and only if $F$ is Lagrangian and 
\begin{align}\label{TLS}
\theta_{L} - \langle F, JT\rangle = \theta_{0}
\end{align}
for some constant $\theta_{0}\in\mathbb{R}$. We will call the constant $\theta_{0}$ the {\it phase} of $F$. This terminology is motivated by the following observation in \cite{WBfmin}. Let   
\begin{align*}
\Omega_{f}:=e^{-\frac{f}{2} - i\langle z, JT \rangle}\Omega,
\end{align*}
where $f(z) = 2\langle z, T\rangle$. Then $\Omega_{f}$ is a holomorphic $(m, 0)$-form on $\mathbb{C}^{m}$ and satisfies $(\ref{ACY})$. Hence $(\mathbb{C}^{m}, J, \omega_{0}, \:\Omega_{f})$ is an {\it almost Calabi--Yau $m$-fold}. By Lagrangian condition we then have
\begin{align*}
F^{*}(\Omega_{f}) = e^{i(\theta(F) - \langle F, JT\rangle)}e^{-\frac{F^{*}f}{2}}dV_{g} = e^{i\theta_{0}}(e^{-\frac{F^{*}f}{2}}dV_{g}).
\end{align*}
Therefore Lagrangian translating solitons with phase $\theta_{0}$ are exactly the $f$-special Lagrangians  with phase $\theta_{0}$.

Given $T\in\mathbb{C}^{m}$, we will distinguish the submanifolds $F:L\to\mathbb{C}^{m}$ into two classes.
\begin{definition}\label{Tfinite}
Given $T\in\mathbb{C}^{m}$, we say a submanifold $F:L\to\mathbb{C}^{m}$ is $T$-finite if
\begin{align*}
    \langle F(x), T\rangle< +\infty,
\end{align*}
where $\langle\cdot, \cdot\rangle$ is the Euclidean metric on $\mathbb{C}^{m}$.
Otherwise, we say $F:L\to\mathbb{C}^{m}$ is $T$-infinite.
\end{definition}
We remark that if a Lagrangian submanifold  $F:L\to\mathbb{C}^{m}$ is $T$-finite, then the induced $f$-volume $e^{-\frac{F^{*}f}{2}}dV_{g} = e^{-\frac{2\langle F(x), T\rangle}{2}}dV_{g}$ has a {\it positive lower bound}.

\section{Building Blocks}
The building blocks of our gluing construction lie in the categories of Lagrangian translating solitons with conical singularities (CS) and asymptotically conical (AC) Lagrangian submanifolds. We first introduce the notion of CS and AC ends of a Riemannian manifold.

The underlying topological manifolds are of the following type.
\begin{definition}\label{cyl}
Let $X$ be a smooth $m$-manifold. We say $X$ is a manifold with cylindrical ends if there exists a compact subset $K\subset X$, a real number $S>0$, a smooth $(m-1)$-manifold $\Sigma$ without boundary, and a diffeomorphism $\phi$ so that $\phi:\Sigma\times(S, \infty)\stackrel{\simeq}{\to} X\setminus K$. Suppose $\Sigma$ is decomposed into connected components $\Sigma = \bigcup_{i=1}^{e}\Sigma_{i}$, then the corresponding component $E_{i}:=\phi(\Sigma_{i}\times [S, \infty))$ is called an end of $X$, and $\Sigma_{i}$ will be called the link of $E_{i}$.
\end{definition}

The Riemannian structure on the ends of these manifolds will be modelled on the {\it Riemannian cones}.
\begin{definition}
Let $(\Sigma, g_{\Sigma})$ be a smooth, compact, Riemannian $(m-1)$-manifold. Then the Riemannian cone over $\Sigma$ is the Riemannian $m$-manifold $C_{\Sigma} := (\Sigma\times(0, \infty), g_{C_{\Sigma}})$, where
\begin{align*}
    g_{C_{\Sigma}}(r, \sigma) := dr^{2} + r^{2}g_{\Sigma}(\sigma).
\end{align*}

\end{definition}

\subsection{CS and AC ends}

\begin{definition}\label{RiemACCS}
Let $(X, g)$ be a Riemannian manifold with $X$ being a manifold with cylindrical ends as in Definition $\ref{cyl}$, and let $E_{i}$ be one of the ends of $X$ with link $\Sigma_{i}$, $i = 1, \cdots, e$. For each $i$, let $C_{i} = (\Sigma_{i}\times (0, \infty), g_{C_{i}})$ be the Riemannian cone over $\Sigma_{i}$. 
\begin{description}
    \item[1.] We say $E_{i}$ is a conical singular (CS) end of $X$ with cone $C_{i}$ and rate $\mu_{i}$ if there exist $\check{\epsilon}_{i}>0$ and a diffeomorphism $\check{\phi}_{i}: \Sigma_{i}\times(0, \check{\epsilon}]\to E_{i}$ such that for all $k\geq 0$,
\begin{align*}
    |\nabla^{k}(\check{\phi}_{i}^{*}g - g_{C_{i}})| = O(r^{\mu_{i} - k})
\end{align*}
as $r\to 0$. 
    \item[2.] We say $E_{i}$ is an asymptotically conical (AC) end of $X$ with cone $C_{i}$ and rate $\lambda_{i}$ if there exist $\hat{R}_{i}>0$ and a diffeomorphism $\hat{\phi}_{i}:[\hat{R}_{i}, \infty)\times\Sigma_{i}\to E_{i}$ such that for all $k\geq 0$,
\begin{align*}
    |\nabla^{k}(\hat{\phi}^{*}g - g_{C_{i}})| = O(r^{\lambda_{i} - k})
\end{align*}
as $r\to \infty$. 
\item[3.] If all the ends of $X$ are CS (resp. AC), then $X$ is called a CS (resp. AC) Riemannian manifold.
\end{description}
Here, the connection $\nabla$ and the norm $|\cdot|$ are computed using the cone metric $g_{C_{i}} = dr^{2} + r^{2}g_{\Sigma_{i}}$.
\end{definition}

The CS and AC Riemannian structures that we will consider in this paper are induced from the {\it Lagrangian submanifolds with conical singularities} and {\it asymptotically conical Lagrangian subamnifolds} in $\mathbb{C}^{m}$. 

Let $\Sigma^{m-1}\subset S^{2m-1}$ be a Legendrian submanifold in the standard unit sphere $S^{2m-1}\subset \mathbb{C}^{m}$. Then the immersion $\iota_{C}: \Sigma\times (0, \infty)\to\mathbb{C}^{m}$ given by $\iota_{C}(\sigma, r) := r\sigma$ is a Lagrangian cone. We will denote $C := \iota_{C}(\Sigma\times (0, \infty))$. The submanifold $\Sigma\subset S^{2m-1}$ will be called the {\it link} of $C$. 

\begin{definition}
Let $L$ be a manifold with cylindrical ends $E_{1}, \cdots, E_{e}$, and $F:L\to\mathbb{C}^{m}$ be a Lagrangian submanifold. Suppose we have Lagrangian cones $C_{i} = \iota_{C_{i}}(\Sigma_{i}\times (0, \infty))$, $i = 1, \cdots, e$, in $\mathbb{C}^{m}$. 
\begin{itemize}
    \item[1.] We say $F:L\to\mathbb{C}^{m}$ is a Lagrangian submanifold with conical singularities $p_{1}, \cdots, p_{e}\in\mathbb{C}^{m}$ with cones $C_{1}, \cdots, C_{e}$ and rates $\mu_{1}, \cdots, \mu_{e}$ if for each $i = 1, \cdots, e$ there exist a $\check{\epsilon}_{i}>0$ and a diffeomorphism $\check{\phi}_{i}: \Sigma\times(0, \check{\epsilon}_{i}]\to E_{i}$ such that for all $k\geq 0$,
    \begin{align*}
        |\nabla^{k}(F\circ\check{\phi}_{i} - (\iota_{C_{i}} + p_{i}))| = O(r^{\mu_{i} - 1 - k})
    \end{align*}
    as $r\to 0$.
    \item[2.] We say $F:L\to\mathbb{C}^{m}$ is an asymptotically conical Lagrangian submanifold with cones $C_{1}, \cdots, C_{e}$, rates $\lambda_{1}, \cdots, \lambda_{e}$, and centers $q_{1}, \cdots, q_{e}$, if for each $i = 1, \cdots, e$ there exist $\hat{R}_{i}>0$ and a diffeomorphism $\hat{\phi}_{i}: \Sigma\times[\hat{R}_{i}, \infty]\to E_{i}$ such that for all $k\geq 0$,
    \begin{align*}
        |\nabla^{k}(F\circ\hat{\phi}_{i} - (\iota_{C_{i}} + q_{i}))| = O(r^{\lambda_{i} - 1 - k})
    \end{align*}
    as $r\to \infty$.
\end{itemize}
Here, the connection $\nabla$ and the norm $|\cdot|$ are computed using the induced cone metric $g_{C_{i}} = dr^{2} + r^{2}g_{\Sigma_{i}}$.
\end{definition}
For brevity, we will call Lagrangian submanifolds with conical singularities and asymptotically conical Lagrangian submanifolds {\it CS} and {\it AC} Lagrangian submanifolds, respectively. It is not hard to see if $F:L\to\mathbb{C}^{m}$ is a CS/AC Lagrangian submanifold, then $L$ together with the induced metric is a Riemannian manifold with CS/AC ends.

The Lagrangian submanifolds {\it with both AC and CS ends} are defined in an obvious manner. Such Lagrangian submanifolds in $\mathbb{C}^{m}$ are called {\it Lagrangian conifolds}. It is also possible that some ends of $L$ are neither CS nor AC ends.

\subsection{Lawlor necks}
We describe an important example of asymptotically conical special Lagrangians in $\mathbb{C}^{m}$ found by Lawlor \cite{Lawlor}, which will be used later to resolve conical singular points modeled on transversely intersecting planes.

Given Lagrangian planes $\Pi_{0}$, $\Pi_{1}$ in $\mathbb{C}^{m}$ so that $\Pi_{0}\cap\Pi_{1} = \{0\}$. Then up to an action of $U(m)$, we can put $\Pi_{0}$, $\Pi_{1}$ into the form
\begin{align*}
\Pi_{0} = \{(x_{1}, \cdots, x_{m})\:|\:x_{j}\in\mathbb{R}\},\quad\Pi_{1} = \{(x_{1}e^{i\phi_{1}}, \cdots, x_{m}e^{i\phi_{m}})\:|\:x_{j}\in\mathbb{R}\},
\end{align*}
where $\phi := (\phi_{1}, \cdots, \phi_{m})\in (0, \pi)^{m}$. 
\begin{definition}
We say the pair of Lagrangian planes $\Pi_{0}\cup\Pi_{1}$ satisfies the angle condition if
\begin{align*}
\phi_{1} +\cdots + \phi_{m} = \pi.
\end{align*}
\end{definition}
If $\Pi_{0}\cup\Pi_{1}$ satisfies the angle condition, then $\Pi_{0}\cup-\Pi_{1}$ is a special Lagrangian cone in $\mathbb{C}^{m}$. Lawlor \cite{Lawlor} found a family of nonsingular, asymptotically conical special Lagrangians $N: S^{m-1}\times\mathbb{R}\to\mathbb{C}^{m}$ resolving $\Pi_{0}\cup-\Pi_{1}$, which we are going to describe.

Let $m>2$. Given $a_{1}, \cdots, a_{m}>0$, define polynomials
\begin{align*}
p(x) = (1+a_{1}x^{2})\cdots(1+a_{m}x^{2}),\quad P(x) = \frac{p(x)}{x^{2}},
\end{align*}
and define $\phi_{1}, \cdots, \phi_{m}, A$ by
\begin{align*}
\phi_{j} = a_{j}\int_{-\infty}^{\infty}\frac{dx}{(1+a_{j}x^{2})\sqrt{P(x)}}, \quad A = \int_{-\infty}^{\infty}\frac{dx}{2\sqrt{P(x)}}.
\end{align*}
Note that $\phi_{1}+\cdots+\phi_{m} = \pi$, and this gives an $1$-$1$ correspondence between the $m$-tuples $(a_{1}, \cdots, a_{m})\in\mathbb{R}_{>0}$ and $(m+1)$-tuples $(\phi_{1}, \cdots, \phi_{m}, A)\in(0, \pi)^{m}\times\mathbb{R}_{>0}$ with $\phi_{1}+\cdots+\phi_{m} = \pi$. 

Define functions $\psi_{1}(y), \cdots, \psi_{m}(y)$ by
\begin{align*}
\psi_{j}(y) = a_{j}\int_{-\infty}^{y}\frac{dx}{(1+a_{j}x^{2})\sqrt{P(x)}},
\end{align*}
and functions $z_{j}:\mathbb{R}\to\mathbb{C}$ by $z_{j}(y):=e^{i\psi_{j}(y)}\sqrt{a^{-1}_{j}+ y^{2}}$ for $j=1, \cdots, m$. Now define a submanifold
\begin{align*}
N_{\phi, A} : S^{m-1}\times\mathbb{R}\to\mathbb{C}^{m}, \quad N_{\phi, A}(x_{1}, \cdots, x_{m}, y) := (x_{1}z_{1}(y), \cdots, x_{m}z_{m}(y)),
\end{align*}
where $x^{2}_{1}+\cdots x^{2}_{m} = 1$. Then $N_{\phi, A}$ defines an embedded, asymptotically conical special Lagrangian submanifold in $\mathbb{C}^{m}$, called the {\it Lawlor neck}. Imagi--Joyce--Oliveria dos Santos \cite{IJO} shows that the Lawlor neck is the unique smooth, exact, AC special Lagrangian with cone $C = \Pi_{0}\cup\Pi_{1}$ and rate $\lambda<2$.

\subsection{Grim Reaper cylinders}
A curve $\gamma: (-\pi/2, \pi/2)\to\mathbb{C}$ given by
\begin{align*}
    \gamma(s) := -\log\cos(s) - is
\end{align*}
is called the {\it Grim Reaper curve} (see Figure 2). It is a translating soliton of curve shortening flow, hence an example of translating soliton to LMCF for $m=1$. It is easy to see that the product immersion
\begin{align*}
    F_{\gamma}:\mathbb{R}^{m-1}\times(-\pi/2, \pi/2)\to\mathbb{C}^{m},\quad F_{\gamma}(x_{1}, \cdots, x_{m-1}, x_{m}) := (x_{1}, \cdots x_{m-1}, \gamma(x_{m}))
\end{align*}
defines a $T$-finite Lagrangian translating soliton in $\mathbb{C}^{m}$ diffeomorphic to $\mathbb{R}^{m}$, with $T = -e_{m}$, where $e_{m} = (0, \cdots, 0, 1)$. We will still call $F_{\gamma}$ a Grim Reaper cylinder in $\mathbb{C}^{m}$. We now show that $F_{\gamma}$ is an $f$-special Lagrangian with $f(z) := 2\langle z, -e_{m} \rangle$, where $e_{m} = (0, \cdots, 0, 1)\in\mathbb{C}^{m}$, and compute its phase.

Since $\gamma'(x_{m}) = \tan(x_{m}) - i$, the Lagrangian angle $\theta(F_{\gamma})$ of $F_{\gamma}$ with respect to $\Omega$ is given by $\arg(\tan(x_{m}) - i) = x_{m} - \frac{\pi}{2}$. On the other hand, $\langle F_{\gamma}, J(-e_{m})\rangle = x_{m}$. Thus we have
\begin{align*}
    \theta(F_{\gamma}) - \langle F_{\gamma}, J(-e_{m})\rangle = -\frac{\pi}{2},
\end{align*}
so $F_{\gamma}$ is an $f$-special Lagrangian with phase $-\frac{\pi}{2}$. If we translate $F_{\gamma}$ by $\phi\in\mathbb{R}$ in $\im(z_{m})$-direction, that is, set $F^{\phi}_{\gamma}:=F_{\gamma} + (0, \cdots, 0, i\phi)$, then the Lagrangian angle is unchanged and $\langle F^{\phi}_{\gamma}, J(-e_{m})\rangle = x_{m} - \phi$. Therefore $F^{\phi}_{\gamma}$ is $f$-special Lagrangian with phase $\phi - \frac{\pi}{2}$. In particular, $F^{\frac{\pi}{2}}_{\gamma}$ has phase $0$.

\begin{center}
    \includegraphics[width=0.5\textwidth]{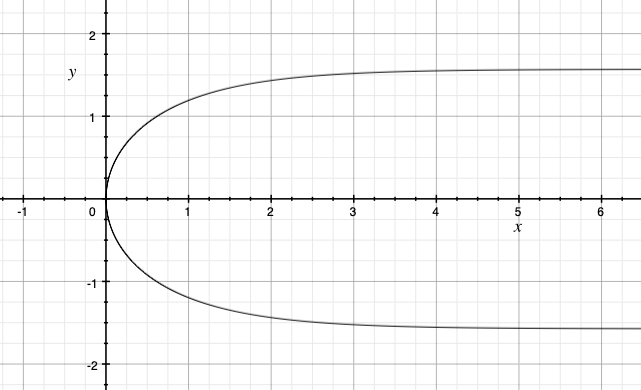}\\
    \small{Figure 2. a Grim Reaper curve $\gamma$}
\end{center}

\subsection{Lagrangian translating solitons with isolated conical singularities}
We now construct examples of Lagrangian translating solitons with isolated conical singularities by intersecting Grim Reaper cylinders with same phase. The tangent cones at each singularity is a pair of special Lagrangian planes with the same phase. This property allows us to use the Lawlor neck to resolve the singularities.

\ 

Consider the Grim Repaer $F_{0} := F^{\frac{\pi}{2}}_{\gamma}$ with phase $0$. Given $\phi = (\phi_{1}, \cdots, \phi_{m})\in (0, \pi)^{m}$ and $\lambda\in\mathbb{R}$,  define 
\begin{align*}
    &F^{\lambda}_{\phi} := P_{(\phi, \lambda)}\circ F_{0}: \mathbb{R}^{m-1}\times(-\pi/2, \pi/2)\to\mathbb{C}^{m},\\
    &F^{\lambda}_{\phi}(x_{1}, \cdots, x_{m}) = \left( x_{1}e^{i\phi_{1}}, \cdots, x_{m-1}e^{i\phi_{m-1}}, \lambda - \log\cos(x_{m}) - i(x_{m} - \pi/2) + i\phi_{m}\right),
\end{align*}
where $P_{\phi, \lambda}$ is the affine transformation defined in $(\ref{AffineTrans})$.
Then $L^{\lambda}_{\phi} := F^{\lambda}_{\phi}(\mathbb{R}^{m-1}\times(-\pi/2, \pi/2))$ is a Lagrangian translating soliton with phase $\sum_{j=1}^{m}\phi_{j}$. Thus, if $\phi\in\mathcal{A}$ is admissible (see $(\ref{AdmissAngles})$), that is, $\sum_{j=1}^{m}\phi_{j} = \pi$, then $-L^{\lambda}_{\phi}$ is a Lagrangian translating soliton with phase $0$, where we use $-L^{\lambda}_{\phi}$ to denote the same submanifold as $L^{\lambda}_{\phi}$ but with reversed orientation.

As $\phi_{m}\in (0, \pi)$, there is a unique intersection point $\{p\} = L_{0}\cap L^{\lambda}_{\phi}$, where $p\in\mathbb{C}^{m}$ is of the form $(0, \cdots, 0, q)$, $q\in\mathbb{C}$. Let $s_{0}, s_{1}\in (-\pi/2, \pi/2)$ be the points such that 
\begin{align*}
q = -\log\cos(s_{0}) - i(s_{0} - \pi/2) = \lambda-\log\cos(s_{1}) - i(s_{1}-\pi/2) + i\phi_{m}.
\end{align*}
Then $s_{1} - s_{0} = \phi_{m}$. We compute the characterizing angles of the tangent cone of $L_{0}\cup L^{\lambda}_{\phi}$ at $p$. Let $p_{0}$, $p^{\lambda}_{\phi}$ be the preimages of $p$ under $F_{0}$ and $F^{\lambda}_{\phi}$, respectively.
\begin{lemma}
For any $\lambda\in\mathbb{R}$, the characterizing angles of $T_{p_{0}}L_{0}\cup T_{p^{\lambda}_{\phi}}L^{\lambda}_{\phi}$ are $(\phi_{1}, \cdots, \phi_{m-1}, \phi_{m})$, which do not depend on $\lambda$. Hence it satisfies the angle condition if and only if $\phi\in\mathcal{A}$.
\end{lemma}
\begin{proof}
It suffices to compute the difference between the arguments of the tangent vectors of $x_{m}\mapsto -\log\cos(x_{m}) - i(x_{m} - \pi/2)$ and $x_{m}\mapsto \lambda-\log\cos(x_{m}) - i(x_{m} - \pi/2) + i\phi_{m}$. It is clear that the tangent vectors at $x_{m} = s_{0}$ and $x_{m} = s_{1}$ are $\tan(s_{0}) - i$ and $\tan(s_{1}) - i$, respectively. Therefore, the difference of the arguments is
\begin{align*}
    \left(s_{1} - \frac{\pi}{2}\right) - \left(s_{0} - \frac{\pi}{2}\right) = s_{1} - s_{0} = \phi_{m}.
\end{align*}
\end{proof}
The resulting submanifold $L_{0}\cup L^{\lambda}_{\phi}$ is a Lagrangian translating soliton with phase $0$ and conical singularity at the intersection point $p\in\mathbb{C}^{m}$ with rate $\mu = 3$ and special Lagrangian cone $T_{p_{0}}L_{0}\cup T_{p^{\lambda}_{\phi}}L^{\lambda}_{\phi}$. Note that $L_{0}\cup L^{\lambda}_{\phi}$ is $T$-finite in the sense of Definition $\ref{Tfinite}$, with $T = -e_{m}$. One can construct more complicated $T$-finite examples of Lagrangian translating solitons with phase $0$ and isolated conical singularities so that each tangent cone satisfies the angle condition by intersecting more $L^{\lambda}_{\phi}$ with different choices of $\lambda\in\mathbb{R}$ and $\phi = (\phi_{1}, \cdots, \phi_{m})\in\mathcal{A}$, see Figure 3.
\begin{center}
    \includegraphics[width=0.45\textwidth]{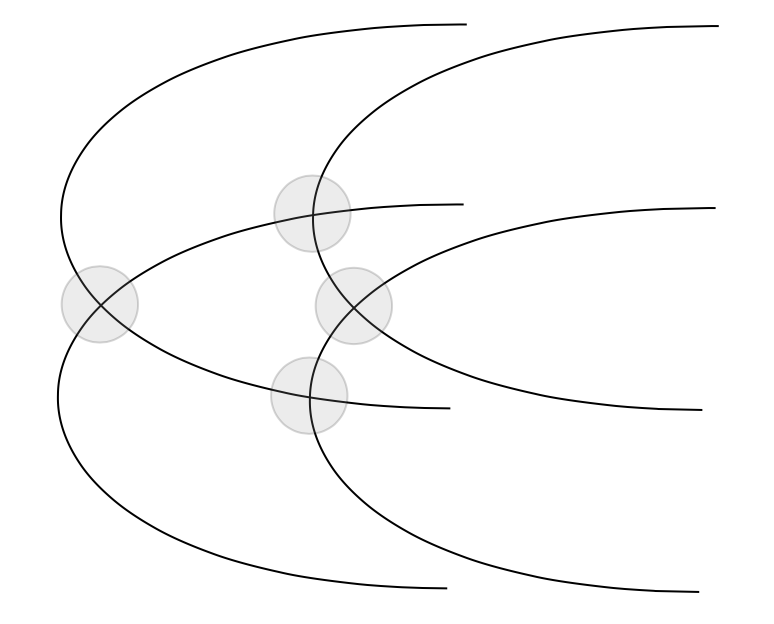}\\
    \small{Figure 3. $T$-finite Lagrangian translating soliton with 4 singularities}
\end{center}

\begin{remark}
One can construct {\it $T$-infinite} examples by intersecting Lagrangian planes and Grim Reaper cylinders. See Figure 4.
\begin{center}
    \includegraphics[width=0.45\textwidth]{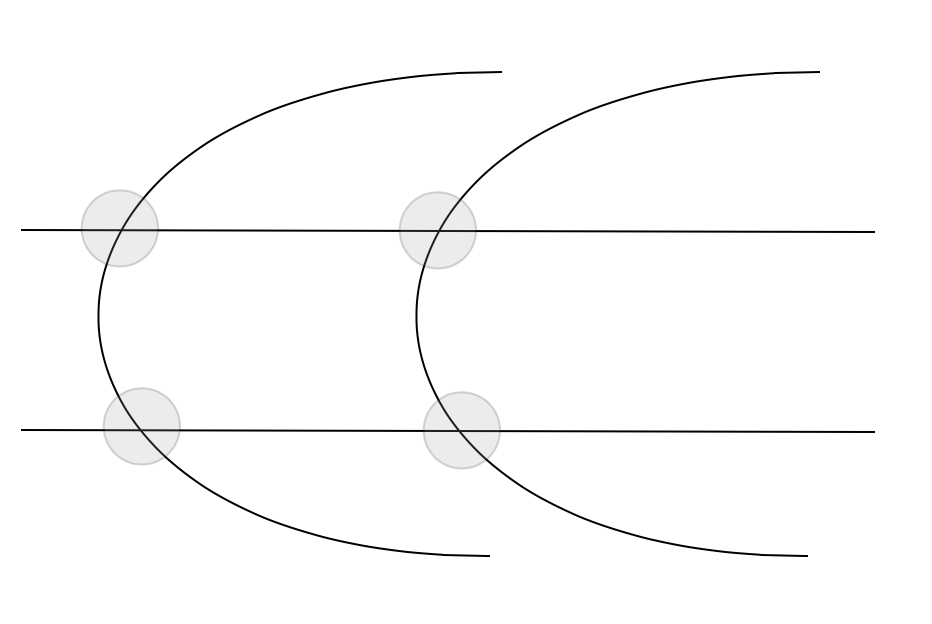}\\
    \small{Figure 4. $T$-infinite Lagrangian translating soliton with 4 singularities}
\end{center}

\end{remark}

\section{Construction of Approximate Solutions}
We will construct approximate solutions to the Lagrangian translating soliton equation $(\ref{translatoreq})$ by resolving the conical singularities of the Lagrangian translating soliton with isolated conical singularities constructed in the last section using special Lagrangian Lawlor necks. The resulting submanifolds are Lagrangian, but not $f$-special Lagrangian, as the Lawlor necks are not $f$-special Lagrangian. 

\subsection{Lagrangian surgery using Lawlor necks}
In this subsection we construct a $1$-parameter family of approximate solutions $\widetilde{F}_{t}:X_{t}\to\mathbb{C}^{m}$ which converges as $t\to 0$ to a Lagrangian translating soliton with isolated singularities $\widetilde{F}$ constructed in \S 3.4. We will only demonstrate the construction for the case where $\widetilde{F}$ is a union of two Grim Reaper cylinders. The constructions for the other cases described in \S3.4 are completely analogous.

Given $\phi = (\phi_{1}, \cdots, \phi_{m})\in\mathcal{A}$ and $\lambda\in\mathbb{R}$. Let $F_{0}$, $F^{\lambda}_{\phi}$ denote the Grim Reaper cylinders as in \S3.4. By parametrizing Grim Reaper curve by the arc-length, we may assume that $F_{0}$ and $F^{\lambda}_{\phi}$  are embeddings of $\mathbb{R}^{m}$ with induced flat metric. To distinguish the domains of $F_{0}$ from that of $F^{\lambda}_{\phi}$, we denote $F_{0}:\mathbb{R}^{m}_{0}\to\mathbb{C}^{m}$ and $F^{\lambda}_{\phi}:\mathbb{R}^{m}_{\phi}\to\mathbb{C}^{m}$, where $\mathbb{R}^{m}_{0}= \mathbb{R}^{m}_{\phi}=\mathbb{R}^{m}$. Define $\widetilde{F}:\mathbb{R}^{m}_{0}\cup\mathbb{R}^{m}_{\phi}\to\mathbb{C}^{m}$ by
\begin{align*}
\widetilde{F}(x) = \left\{\begin{array}{lr}
                  F_{0}(x),&\quad x\in\mathbb{R}^{m}_{0},\\
                  F^{\lambda}_{\phi}(x),&\quad x\in\mathbb{R}^{m}_{\phi}.
                  \end{array}\right.
\end{align*}
Then we have seen from \S3.4 that $\widetilde{F}$ defines a Lagrangian translating soliton with phase $0$, and with an isolated conical singularity at the intersection point $p$ with rate $\mu=3$ and cone $C_{p} = T_{p_{0}}L_{0}\cup -T_{p_{\phi}}L_{\phi}$ satisfying the angle condition.

We now construct a family of embedded Lagrangian submanifolds $\tilde{F}_{t}: X_{t}\to\mathbb{C}^{m}$ for $t\in (0, \delta)$ so that $\tilde{L}_{t} = \tilde{F}_{t}(X_{t})\to \widetilde{L} = \widetilde{F}(\mathbb{R}^{m}_{0}\cup\mathbb{R}^{m}_{\phi})$ as currents as $t\to 0$, by gluing in the special Lagrangian Lawlor necks $N_{t} := tN: S^{m-1}\times\mathbb{R}\to\mathbb{C}^{m}$ asymptotic to $C$, where $X_{t}$ is homeomorphic to the connected sum $\mathbb{R}^{m}_{0}\#\mathbb{R}^{m}_{\phi}\simeq S^{m-1}\times\mathbb{R}$ for each $t\in(0, \delta)$. Let $\Sigma := S^{m-1}\cup S^{m-1}\subset S^{2m-1}$ be the link of $C$. Since $p$ is a conical singularity of $L_{0}\cup L_{\phi}$, there exists $\epsilon>0$ and a diffeomorphism $\check{\phi}: \Sigma\times (0, \epsilon]\to E_{p}:=\overline{B_{1}(p_{0})}\setminus\{p_{0}\}\cup \overline{B_{1}(p_{\phi})}\setminus\{p_{\phi}\}$, where $B_{1}(p_{0})$ and $B_{1}(p_{\phi})$ are the unit balls in $\mathbb{R}^{m}_{0}$, $\mathbb{R}^{m}_{\phi}$, repectively, so that for every $k\in\{0\}\cup\mathbb{N}$,
\begin{align*}
    |\nabla^{k}(\widetilde{F}\circ\check{\phi} - (\iota_{C} + p))|_{g_{C}} = O(r^{2-k})\quad\mbox{as $r\to 0$}.
\end{align*}
Since $\mu = 3>0$, by \cite[p239-240]{PaciniGlue}, there exists $u\in C^{\infty}(\Sigma\times(0, \epsilon])$ with $|\nabla^{k}u|_{g_{C}} = O(r^{3-k})$ for each $k\in\{0\}\cup\mathbb{N}$, so that $\widetilde{F}(E_{p})$ can be written as a graph
\begin{align*}
    \widetilde{F}\circ\check{\phi}(\Sigma\times(0, \epsilon]) = \{\: p + \iota_{C}(\sigma, r) + J(\iota_{C})_{*}(\nabla u)(\sigma, r)\:|\:(\sigma, r)\in \Sigma\times(0, \epsilon]\:\}
\end{align*}
over $C$. Likewise, the Lawlor neck $N: S^{m-1}\times\mathbb{R}\to\mathbb{C}^{m}$ is AC with rate $\lambda = 2-m$ and cone $C$, so there exist a compact subset $K\subset S^{m-1}\times\mathbb{R}$, a constant $\hat{R}>0$ and a diffeomorphism $\hat{\phi}:\Sigma\times [\hat{R}, \infty)\to \hat{E}_{\infty}:=\overline{(S^{m-1}\times\mathbb{R})\setminus K}$ so that 
\begin{align}\label{lawlordiff}
    |\nabla^{k}(N\circ\hat{\phi} - \iota_{C})|_{g_{C}} = O(r^{1-m-k})\quad\mbox{ as $r\to\infty$}.
\end{align}
For each $t\in(0, \delta)$, consider the rescaled Lawlor neck $tN: S^{m-1}\times\mathbb{R}\to\mathbb{C}^{m}$. Define $\hat{\phi}_{t}(\sigma, r):=\hat{\phi}(\sigma, r/t)$, then $\hat{\phi}_{t}: \Sigma\times[t\hat{R}, \infty)\to \hat{E}_{\infty}$ is a diffeomorphism, and
\begin{align*}
    |\nabla^{k}(tN\circ\hat{\phi}_{t} - \iota_{C})|_{g_{C}} = t^{m}O(r^{1-m-k})\quad\mbox{ as $r\to\infty$}.
\end{align*}
Hence $tN$ is again AC with rate $\lambda = 2-m$ and cone $C$. We can find $v_{t}\in C^{\infty}(\Sigma\times[t\hat{R}, \infty))$ so that $tN(\hat{E}_{\infty}) + p$ can be written as a graph
\begin{align*}
    tN(\hat{E}_{\infty}) + p = \{\:p + \iota_{C}( \sigma, r) + J(\iota_{C})_{*}(\nabla v_{t})(\sigma, r)\:|\:(\sigma, r)\in\Sigma\times[t\hat{R}, \infty)  \:\}
\end{align*}
over $C$. 

Choose $\tau\in (0, 1)$. Then for $\delta$ small enough we have $t\hat{R}<t^{\tau}<2t^{\tau}<\epsilon$ for $t\in (0, \delta)$. Fix a smooth monotone increasing function $\eta:(0, \infty)\to [0, 1]$ such that $\eta(r)\equiv 0$ on $(0, 1]$ and $\eta(r)\equiv 1$ on $[2, \infty)$. Define $w_{t}:\Sigma\times[t\hat{R}, \epsilon]\to\mathbb{R}$ by
\begin{align*}
w_{t}(\sigma, r) := \eta(t^{-\tau}r)u(\sigma, r) + (1-\eta(t^{-\tau}r))v_{t}(\sigma, r).
\end{align*}
Then $w_{t}$ is a smooth interpolation between $u$ and $v_{t}$. Define an $m$-manifold $X_{t}$ by
\begin{align*}
X_{t} := [\:(\mathbb{R}^{m}_{0}\cup\mathbb{R}^{m}_{\phi})\setminus E_{p}\:]\cup [\: \sigma\times [t\hat{R}, \epsilon]\:]\cup [\:(S^{m-1}\times\mathbb{R})\setminus \hat{E}_{\infty}\:].
\end{align*}
and an embedding $\tilde{F}_{t}: X_{t}\to\mathbb{C}^{m}$ by
\begin{align*}
\widetilde{F}_{t}(x) = \left\{\begin{array}{ll}
                          \widetilde{F}(x),&\quad x\in (\mathbb{R}^{m}_{0}\cup\mathbb{R}^{m}_{\phi})\setminus E_{p},\\
                          p + \iota_{C}(x) + J(\iota_{C})_{*}(\nabla w_{t})(x),&\quad x\in\Sigma\times[t\hat{R}, \epsilon],\\
                          p+tN(x),&\quad x\in (S^{m-1}\times\mathbb{R})\setminus \hat{E}_{\infty},
                         \end{array}\right.
\end{align*}
Then by construction, $\tilde{F}_{t}:X_{t}\to\mathbb{C}^{m}$ is a smooth, embedded, Lagrangian submanifold in $\mathbb{C}^{m}$. We will call $\widetilde{F}_{t}$ an {\it approximate solution}. The induced metric $g_{t}:= \widetilde{F}_{t}^{*}g_{0}$ on $X_{t}$ satisfies the following property.
\begin{lemma}[\cite{PaciniGlue}, Lemma 4.3]\label{metric}
Let $g$ be the Euclidean metric on $\mathbb{R}^{m}$, and let $g^{N}_{t}$ be the induced AC metric of $tN: S^{m-1}\times\mathbb{R}\to\mathbb{C}^{m}$.  Assume that $\tau\in (\frac{m}{m+1}, 1)$. Then $g_{t}$ satisfies
\begin{align*}
    g_{t} = \left\{\begin{array}{ll}
                          g,&\quad x\in (\mathbb{R}^{m}_{0}\cup\mathbb{R}^{m}_{\phi})\setminus E_{p},\\ \phi^{*}g,&\quad x\in\Sigma\times[2t^{\tau}, \epsilon],\\ \hat{\phi}_{t}^{*}g^{N}_{t},&\quad x\in\Sigma\times[t\hat{R}, t^{\tau}],\\
                          g^{N}_{t},&\quad x\in (S^{m-1}\times\mathbb{R})\setminus \hat{E}_{\infty},
                         \end{array}\right.
\end{align*}
and for all $k\geq 0$,
\begin{align*}
    \sup_{\Sigma\times[t^{\tau}, t^{2\tau}]}|\nabla^{k}(g_{t} - g_{C})|_{r^{-2}g_{C}\otimes g_{C}}\to 0\quad\mbox{ as $t\to 0$}.
\end{align*}
\end{lemma}

The components $(\mathbb{R}^{m}_{0}\cup\mathbb{R}^{m}_{\phi})\setminus E_{p} = [\:\mathbb{R}^{m}_{0}\setminus B_{1}(p_{0})\:]\cup[\:\mathbb{R}^{m}_{\phi}\setminus B_{1}(p_{\phi})\:]$ are noncompact and equipped with Euclidean metric $g$. Following \cite{Trident}, we will call them the {\it wings}. The component $(S^{m-1}\times\mathbb{R})\setminus \hat{E}_{\infty}$ is the truncated cylinder with an induced AC metric $g^{N}_{t}$ and will be called the {\it neck region}. The {\it transition region} $\Sigma\times[t\hat{R}, \epsilon]$ is a smooth interpolation between the wings and the neck. We remark that if one construct approximate solution of more than one isolated conical singularities, than the wings consist of more components of $\mathbb{R}^{m}$ punctured by more balls. For instance, the wings of the approximate solution constructed from Figure 3 or Figure 4 would consist of four copies of $\mathbb{R}^{m}$ punctured by two balls.

\subsection{Error estimate on the approximate solutions}
The Lagrangian submanifold $\tilde{F}_{t}$ is a translating soliton of phase $0$ when restricting to the wings $(\mathbb{R}^{m}_{0}\cup\mathbb{R}^{m}_{\phi})\setminus E_{p}$. However, since the Lawlor neck is not $f$-special Lagrangian, there are some {\it error terms} on the neck and transition region. The error can be described as follows. Recall that $(\mathbb{C}^{m}, J, \omega_{0}, \Omega_{f})$ is an almost Calabi--Yau $m$-fold, where $\Omega_{f}$ is a holomorphic $(m, 0)$-form given by $\Omega_{f} = e^{-\frac{f}{2}- i\langle z, JT\rangle}dz_{1}\wedge\cdots\wedge dz_{m}$, and Lagrangian translating solitons are $f$-special Lagrangians in $(\mathbb{C}^{m}, J, \omega_{0}, \Omega_{f})$. For any Lagrangian immersion $F: X\to\mathbb{C}^{m}$, consider the function defined by
\begin{align*}
\Theta(F) := *F^{*}\im\Omega_{f}: X\to\mathbb{R},
\end{align*}
where $*$ is the Hodge star with respect to the induced metric $g=F^{*}g_{0}$.
Then $\Theta(F) = 0$ if and only if $F:X\to\mathbb{C}^{m}$ is an $f$-special Lagrangian of phase $0$. Hence, what we want is to estimate $\Theta(\widetilde{F}_{t})$ on the neck and transition region, where $\widetilde{F}_{t}:X_{t}\to\mathbb{C}^{m}$ is the approximate solution constructed in \S 4.1. First, we employ the $C^{1}$-estimate of $\Theta(\widetilde{F}_{t})$ proved by Joyce \cite{Joyce3} adapted to our situation (In our case, we have a conical singularity with rate $\mu=3$ and an AC end with rate $\lambda = 2-m$).
\begin{proposition}[\cite{Joyce3}, Proposition 6.4]\label{initial_estimate}
There exists $\delta>0$ and $C>0$ such that for all $t\in (0, \delta)$ and $(\sigma, r)\in\Sigma\times(t\hat{R}, \epsilon)$,
\begin{align*}
|\Theta(\widetilde{F}_{t})(\sigma, r)| &= \left\{\begin{array}{ll}
                          Cr,&\quad r\in(t\hat{R}, t^{\tau}],\\
                          Ct^{\tau} + Ct^{(1-\tau)m},&\quad r\in(t^{\tau}, 2t^{\tau}),\\
                          0,&\quad r\in[2t^{\tau}, \epsilon),
                          \end{array}\right.\\
|d\Theta(\widetilde{F}_{t})(\sigma, r)| &= \left\{\begin{array}{ll}
                          C,&\quad r\in(t\hat{R}, t^{\tau}],\\
                          C + Ct^{(1-\tau)m-\tau},&\quad r\in(t^{\tau}, 2t^{\tau}),\\
                          0,&\quad r\in[2t^{\tau}, \epsilon),
                          \end{array}\right.\\
    \mbox{and} \quad|\Theta(\widetilde{F}_{t})|\leq Ct,&\quad |d\Theta(\widetilde{F}_{t})|\leq C\quad\mbox{on}\quad (S^{m-1}\times\mathbb{R})\setminus \hat{E}_{\infty}.
\end{align*}
Here, the norm $|\cdot|$ is computed using the induced metrics on the corresponding components.
\end{proposition}
For our purpose, we need to estimate $\Theta(\widetilde{F}_{t})$ in {\it weighted Sobolev spaces}, which we shall now define.

Fix a diffeomorphism $\hat{\phi}: \Sigma\times[\hat{R}, \infty)\to \hat{E}_{\infty}$ for the Lawlor neck $N: S^{m-1}\times\mathbb{R}\to\mathbb{C}^{m}$ as in $(\ref{lawlordiff})$. Let $\hat{r}: S^{m-1}\times\mathbb{R}\to\mathbb{R}_{+}$ to be a smooth function such that $\hat{r}\big|_{\hat{E}_{\infty}}(\hat{\phi}(\sigma, r)) = r$. That is, $\hat{r}$ coincides with the coordinate function $r$ on the AC end $\hat{E}_{\infty}$, under the identification of $\hat{\phi}$. Then for $t\in(0, \delta)$, let  $\rho_{t}: X_{t}\to\mathbb{R}_{+}$ be a smooth function such that
\begin{align*}
\rho_{t}(x) = \left\{\begin{array}{ll}
                          1,&\quad x\in (\mathbb{R}^{m}_{0}\cup\mathbb{R}^{m}_{\phi})\setminus E_{p},\\
                          r, &\quad x = (\sigma, r)\in\Sigma\times[t\hat{R}, \epsilon],\\
                          t\hat{r}(x),&\quad x\in (S^{m-1}\times\mathbb{R})\setminus \hat{E}_{\infty}.
                          \end{array}\right.
\end{align*}
Denote $f_{t}:=2\langle \widetilde{F}_{t}, T\rangle$. For $k\geq 0$, $p>1$, $\beta, \gamma\in\mathbb{R}$, define $W^{k, p}_{\beta, \gamma, t}(X_{t})$ to be the Banach space completion of $C^{\infty}_{c}(X_{t})$ with respect to the norm
\begin{align*}
\|u\|_{W^{k, p}_{\beta, \gamma, t}} := \left(\sum_{j=0}^{k}\int_{X_{t}}|e^{\frac{\beta f_{t}}{2}}\rho^{-\gamma+j}_{t}\nabla^{j}u|^{p}\:\rho^{-m}_{t}\:dV_{g_{t}}\right)^{\frac{1}{p}},
\end{align*}
where $|\cdot|$, $\nabla$ are computed using the induced metric $g_{t}$ on $X_{t}$. Notice that in the neck region the weighted norm $\|\cdot\|_{W^{k, p}_{\beta, \gamma, t}}$ is equivalent to the one used in Lockhart--McOwen \cite{LM}. Also note that $W^{k, p}_{\beta, \gamma, t}$ is a special case of weighted Sobolev spaces defined in \cite[p16]{PaciniUnifEst} with weight $w := e^{\frac{\beta f_{t}}{2}}\rho_{t}^{-\gamma}$.

\begin{proposition}\label{InitialEst}
For any $\gamma<0$, $\beta\in\mathbb{R}$, $\tau\in(\frac{m}{m+1}, 1)$, there exists $C>0$ such that
\begin{align*}
\|\Theta(\widetilde{F}_{t})\|_{W^{1, p}_{\beta+1, \gamma-2, t}}\leq C\:t^{\tau(2-\gamma) + (1-\tau)m}.
\end{align*}
\end{proposition}
\begin{proof}
For convenience, we will denote the constants that appear in the estimates which are not dependent on $t$ by $C$. Hence, the values of $C$ in different inequalities may be different.

Note that $\Theta(\widetilde{F}_{t})$ is supported on $[(S^{m-1}\times\mathbb{R})\setminus \hat{E}_{\infty}]\cup(\Sigma\times [t\hat{R}, 2t^{\tau}])$. We first estimate $\Theta(\widetilde{F}_{t})$ in the compact region $I_{t}:=(S^{m-1}\times\mathbb{R})\setminus \hat{E}_{\infty}$. In $I_{t}$, we have $\rho_{t} = t\hat{r}$ and $\widetilde{F}_{t} = tN$ with induced metric $g_{t} = t^{2}g^{N}$. Also, $\hat{r}$ and $e^{\frac{f_{t}}{2}}$ are bounded. Thus from Proposition $\ref{initial_estimate}$ we have
\begin{align*}
\int_{I_{t}}|\:e^{\frac{(\beta+1)f_{t}}{2}}\rho^{2-\gamma}_{t}\:\Theta(\widetilde{F}_{t})|^{p}\:\rho^{-m}_{t}\:dV_{g_{t}}\leq C\int_{I_{t}}|(t\hat{r})^{2-\gamma}t|^{p}(t\hat{r})^{-m}t^{m}dV_{g^{N}}\leq C\:t^{(3-\gamma)p}.
\end{align*}
Next we estimate $\Theta(\widetilde{F}_{t})$ in the annulus region $II_{t} := \Sigma\times [t\hat{R}, t^{\tau}]$. In $II_{t}$, $\rho_{t}(\sigma, r) = r\in [t\hat{R}, t^{\tau}]$, $\widetilde{F}_{t} = tN$ with induced metric $g_{t}$ equivalent to the cone metric $g_{C}$, and the function $e^{\frac{f_{t}}{2}}$ is bounded.
Thus from Proposition $\ref{initial_estimate}$ we have
\begin{align*}
\int_{II_{t}}|\:e^{\frac{(\beta+1)f}{2}}\rho^{2-\gamma}_{t}\:\Theta(\widetilde{F}_{t})|^{p}\:\rho^{-m}_{t}\:dV_{g_{t}}\leq C\int_{II_{t}}|r^{2-\gamma}r|^{p}r^{-m}dV_{g_{C}}\leq C\int_{t\hat{R}}^{t^{\tau}}r^{(3-\gamma)p-1}dr\leq C\:t^{(3-\gamma)p\tau}.
\end{align*}
In the remaining region $III_{t} := \Sigma\times [t^{\tau}, 2t^{\tau}]$, $\rho_{t} = r\in [t^{\tau}, 2t^{\tau}]$, the induced metric $g_{t}$ is equivalent to the cone metric $g_{C}$ by Lemma $\ref{metric}$, and the function $e^{\frac{f_{t}}{2}}$ is bounded. Thus from Proposition $\ref{initial_estimate}$ we have
\begin{align*}
\int_{III_{t}}|\:e^{\frac{(\beta+1)f_{t}}{2}}\rho^{2-\gamma}_{t}\:\Theta(\widetilde{F}_{t})|^{p}\:\rho^{-m}_{t}\:dV_{g_{t}}&\leq  C\:t^{(2-\gamma)\tau p}(t^{\tau} + t^{(1-\tau)m})^{p}\int_{III_{t}}r^{-m}dV_{g_{C}}\\
&\leq C\:t^{(2-\gamma)\tau p}(t^{\tau } + t^{(1-\tau)m})^{p}.
\end{align*}
Combining the estimates together, we get
\begin{align*}
\|\Theta(\widetilde{F}_{t})\|_{W^{0, p}_{\beta+1, \gamma, t}}\leq C\:[t^{3-\gamma} + t^{(3-\gamma)\tau} + t^{(2-\gamma)\tau + (1-\tau)m}]
\end{align*}

By similar computation,
\begin{align*}
&\int_{I_{t}}|\:e^{\frac{(\beta+1)f}{2}}\rho^{3-\gamma}_{t}\:d\Theta(\tilde{F}_{t})|^{p}\:\rho^{-m}_{t}\:dV_{g_{t}}\leq C\:t^{(3-\gamma)p},\\
&\int_{II_{t}}|\:e^{\frac{(\beta+1)f}{2}}\rho^{3-\gamma}_{t}\:d\Theta(\tilde{F}_{t})|^{p}\:\rho^{-m}_{t}\:dV_{g_{t}}\leq C\:t^{(3-\gamma)\tau p}.
\end{align*}
and
\begin{align*}
\int_{III_{t}}|\:e^{\frac{(\beta+1)f}{2}}\rho^{3-\gamma}_{t}\:d\Theta(\tilde{F}_{t})|^{p}\:\rho^{-m}_{t}\:dV_{g_{t}}\leq C\:t^{(3-\gamma)\tau p + ((1-\tau)m - \tau)p}.
\end{align*}
Combining the estimates together, we find
\begin{align*}
\|d\Theta(\tilde{F}_{t})\|_{W^{0, p}_{\beta+1, \gamma-3, t}}\leq C\: [\:t^{3-\gamma} + t^{(3-\gamma)\tau} + t^{(2-\gamma)\tau + (1-\tau)m}].
\end{align*}
Therefore,
\begin{align*}
    \|\Theta(\widetilde{F}_{t})\|_{W^{1, p}_{\beta+1, \gamma -2, t}} \leq C\:[\:t^{3-\gamma} + t^{(3-\gamma)\tau} + t^{(2-\gamma)\tau + (1-\tau)m}]\leq C\:[\:t^{(3-\gamma)\tau} + t^{(2-\gamma)\tau + (1-\tau)m}].
\end{align*}
For $\tau\in (\frac{m}{1+m}, 1)$, we have $(3-\gamma)\tau> \tau(2-\gamma)+ (1-\tau)m$. Therefore
\begin{align*}
\|\Theta(\widetilde{F}_{t})\|_{W^{1, p}_{\beta+1, \gamma-2, t}}\leq C\:t^{\tau(2-\tau) + (1-\tau)m}
\end{align*}
and the proposition is proved.
\end{proof}

\subsection{Sobolev Embedding Theorem}
In this subsection, we state relevant Sobolev embedding results for our weighted Sobolev spaces on $T$-finite approximate solutions. 
For each $k\geq 0$, $\beta, \gamma\in\mathbb{R}$, we define the weighted $C^{k}$ space to be the space of $C^{k}$ functions on $X_{t}$ such that the $C^{k}_{\beta, \gamma, t}$-norm 
\begin{align*}
\|u\|_{C^{k}_{\beta, \gamma, t}} := \sum_{j=0}^{k}\sup_{X_{t}}|\:e^{\frac{\beta f_{t}}{2}}\rho_{t}^{-\gamma + j}\nabla^{j}u\:|_{g_{t}}
\end{align*}
is finite. These are Banach space of functions on $X_{t}$. Note that our $C^{k}_{\beta, \gamma, t}$ spaces are special cases of weighted $C^{k}$ spaces defined in \cite[p16]{PaciniUnifEst} with weight $w = e^{\frac{\beta f_{t}}{2}}\rho_{t}^{-\gamma}$.

We use \cite[Theorem 5.1]{PaciniUnifEst} to show the following Sobolev embedding theorem.
\begin{theorem}\label{sobolevemb}
Suppose $\widetilde{F}_{t}: X_{t}\to\mathbb{C}^{m}$ is $T$-finite. For $p\geq 1$, $lp>m$, $\beta'\geq\beta$,  $\gamma'\geq \gamma$ and $t\in (0, \delta)$, we have a continuous embedding
\begin{align*}
W^{k+l, p}_{\beta, \gamma, t}(X_{t})\hookrightarrow C^{k}_{\beta', \gamma', t}(X_{t}).
\end{align*}
Moreover, there exists a constant $C>0$ independent of $t$ such that 
\begin{align*}
\|u\|_{C^{k}_{\beta', \gamma', t}}\leq C\cdot t^{\gamma-\gamma'}\|u\|_{W^{k+l, p}_{\beta, \gamma, t}}.
\end{align*}
\end{theorem} 
\begin{proof}
Since $\widetilde{F}_{t}: X_{t}\to\mathbb{C}^{m}$ is $T$-finite, we have $e^{\frac{\beta' f}{2}}\leq Ce^{\frac{\beta f}{2}}$ for some $C = C(\beta'-\beta)>0$. Therefore for each $t$ we have a natural continuous embedding
\begin{align}
    C^{k}_{\beta, \gamma, t}(X_{t})\hookrightarrow C^{k}_{\beta', \gamma', t}(X_{t}).
\end{align}
This embedding is not uniform in $t$. In fact, since
\begin{align*}
    e^{\frac{\beta'f_{t}}{2}}\rho_{t}^{-\gamma'} = e^{\frac{(\beta'-\beta)f_{t}}{2}}\rho_{t}^{-\gamma'+\gamma}e^{\frac{\beta f_{t}}{2}}\rho_{t}^{-\gamma}\leq C\rho_{t}^{-\gamma'+\gamma}e^{\frac{\beta f_{t}}{2}}\rho_{t}^{-\gamma},
\end{align*}
in the neck region we have
\begin{align*}
    e^{\frac{\beta'f_{t}}{2}}\rho_{t}^{-\gamma'}\leq C\hat{r}^{\gamma-\gamma'}t^{\gamma-\gamma'}e^{\frac{\beta f_{t}}{2}}\rho_{t}^{-\gamma}\leq Ct^{\gamma-\gamma'}e^{\frac{\beta f_{t}}{2}}\rho_{t}^{-\gamma},
\end{align*}
where $C$ depends on $\beta'-\beta$ and the lower bound of $\hat{r}$ on the Lawlor neck. The above inequality clearly holds in other regions. This yields
\begin{align*}
    \|u\|_{C^{k}_{\beta', \gamma', t}}\leq Ct^{\gamma-\gamma'}\|u\|_{C^{k}_{\beta, \gamma, t}}
\end{align*}
for some $C>0$ independent of $t$. Composing with the Sobolev embedding \cite[Theorem 5.1]{PaciniUnifEst}
\begin{align*}
    W^{k+l, p}_{\beta, \gamma, t}(X_{t})\hookrightarrow C^{k}_{\beta, \gamma, t}(X_{t}),
\end{align*}
the result follows.
\end{proof}

\section{Inverting the Linearized Operator}
Let $\widetilde{F}_{t}: X_{t}\to\mathbb{C}^{m}$ be the approximate solution that we have constructed in the previous section. Let $g_{t}:=\widetilde{F}^{*}_{t}g_{0}$ be the induced metric, $f_{t}:=\widetilde{F}^{*}_{t}f = 2\langle\widetilde{F}_{t}, -e_{m}\rangle$, and $\theta_{f_{t}}$ be the Lagrangian angle with respect to $\Omega_{f_{t}}$, that is, $\theta_{f_{t}}$ satisfies $\widetilde{F}^{*}_{t}\Omega_{f_{t}} = e^{i\theta_{f_{t}}}\:e^{-f_{t}/2}\:dV_{g_{t}}$.

The goal of this section is to show that there exist small $\delta>0$ and some suitable weights $\beta, \gamma\in\mathbb{R}$ such that for all $t\in (0, \delta)$ and $p>1$, the linear operator
\begin{align*}
\mathcal{L}_{g_{t}}\cdot = e^{-\frac{f_{t}}{2}}(\:\cos\theta_{f_{t}}\:\Delta_{f_{t}}\cdot - \sin\theta_{f_{t}}\langle\nabla\theta_{f_{t}}, \nabla\cdot\rangle_{g_{t}})
\end{align*}
is an isomorphism from $W^{3, p}_{\beta, \gamma, t}(X_{t})$ to $W^{1, p}_{\beta+1, \gamma-2, t}(X_{t})$, and the operator  norm of $\mathcal{L}_{g_{t}}^{-1}$ is bounded independent of $t$. In this case, $\mathcal{L}_{g_{t}}$ will be said to have {\it uniform invertibility}. We first study the invertibility of $\mathcal{L}_{g_{t}}$ on each components of $X_{t}$, namely, the wings and the necks. Then the uniform invertibility of $\mathcal{L}_{g_{t}}$ on the whole space $X_{t}$ is obtained by using suitable cut-off functions argument as in Pacini \cite{PaciniUnifEst}.

\subsection{Analysis on the translating solitons with Euclidean metric}
We first define the corresponding weighted Sobolev spaces on translating solitons, or more generally, $f$-special Lagrangian submanfiolds in gradient steady K\"ahler--Ricci solitons. Let $(M, \overline{\omega}, f)$ be a gradient steady K\"ahler--Ricci soliton and $F:L\to M$ be an $f$-special Lagrangian $m$-fold. By \cite[Proposition 4.5]{WBfmin}, $L$ must be {\it noncompact}.

Define the weighted Sobolev space $\widetilde{W}^{k, p}_{\beta}(L)$ to be the completion of the $C^{\infty}(L)$ with respect to the norm
\begin{align*}
\|u\|_{\widetilde{W}^{k, p}_{\beta}} := \left(\sum_{j=0}^{k}\int_{L}|e^{\frac{\beta f}{2}}\nabla^{j}u|^{p}\:dV_{g}\right)^{\frac{1}{p}}.
\end{align*}
Note that when $\beta = 0$, $\widetilde{W}^{k, p}_{\beta}(L)$ reduces to the ordinary Sobolev space $W^{k, p}(L)$.

The $f$-Laplacian $\Delta_{f} = \Delta_{g} - \frac{1}{2}\langle\nabla f, \cdot\rangle$ extends naturally to a bounded operator 
\begin{align*}
    \Delta_{f}: \widetilde{W}^{k, p}_{\beta}(L)\to \widetilde{W}^{k-2, p}_{\beta}(L),
\end{align*}
and the multiplication by $e^{-\frac{\beta f}{2}}$ gives an isomorphism $e^{-\frac{\beta f}{2}} : W^{k, p}(L)\stackrel{\sim}{\longrightarrow} \widetilde{W}^{k, p}_{\beta}(L)$. Hence we may transform $\Delta_{f}: \widetilde{W}^{k, p}_{\beta}(L)\to \widetilde{W}^{k-2, p}_{\beta}(L)$ into an operator
\begin{align*}
\mathcal{L}_{\beta} := e^{\frac{\beta f}{2}}\circ\Delta_{f}\circ e^{-\frac{\beta f}{2}}: W^{k, p}(L)\to W^{k-2, p}(L),
\end{align*}
where by direct computation,
\begin{align*}
\mathcal{L}_{\beta} = \Delta_{g}\cdot - \left(\frac{1}{2} + \beta\right)\langle\nabla f, \nabla\cdot\rangle + \left(-\frac{\beta}{2}\right)\left(\Delta_{g}f - \frac{1}{2}|\nabla f|^{2} - \frac{\beta}{2}|\nabla f|^{2}\right).
\end{align*}
We estimate the zeroth order term of $\mathcal{L}_{\beta}$.
\begin{lemma}\label{negative}
Suppose the scalar curvature $\overline{R}$ of $\overline{g}$ satisfies $\overline{R}\geq0$. Then for $\beta\in (-1, 0)$, there exists $c = c(\beta)>0$ such that 
\begin{align*}
\left(-\frac{\beta}{2}\right)\left(\Delta_{g}f - \frac{1}{2}|\nabla f|^{2} - \frac{\beta}{2}|\nabla f|^{2}\right)<-c.
\end{align*}
\end{lemma}
\begin{proof}
By the computation in \cite[p17-18]{WBfmin}, we have $2\Delta_{g}f - |\nabla f|^{2} = -\overline{R} - |\overline{\nabla}f|^{2}\leq - |\overline{\nabla}f|^{2}$. Hence for $\beta<0$,
\begin{align*}
\Delta_{g}f - \frac{1}{2}|\nabla f|^{2} - \frac{\beta}{2}|\nabla f|^{2}\leq-\frac{1}{2} |\overline{\nabla}f|^{2} - \frac{\beta}{2}|\nabla f|^{2}\leq-\frac{1}{2}(1+\beta)|\overline{\nabla}f|^{2}.
\end{align*} 
Thus 
\begin{align*}
\left(-\frac{\beta}{2}\right)\left(\Delta_{g}f - \frac{1}{2}|\nabla f|^{2} - \frac{\beta}{2}|\nabla f|^{2}\right)<\left(\frac{\beta}{2}\right)\left(\frac{\beta+1}{2}|\overline{\nabla}f|^{2}\right)
\end{align*}
and the Lemma is proved.
\end{proof}

Now we let $(M, \overline{\omega}, f) = (\mathbb{C}^{m}, \omega_{0}, 2\langle z, -e_{m}\rangle)$, then clearly we have $\overline{R} = 0$. Let $F = F_{0}:\mathbb{R}^{m}_{0}\to\mathbb{C}^{m}$ and $F^{\lambda}_{\phi}:\mathbb{R}^{m}_{\phi}\to\mathbb{C}^{m}$ be the Grim Reaper cylinders defined in \S 3.4. Then the induced metrics $F^{*}_{0}g_{0}$ and $(F^{\lambda}_{\phi})^{*}g_{0}$ are Euclidean. By Lemma $\ref{negative}$ and the standard theory of linear elliptic operators on the Euclidean space (see Krylov {\cite{Krylov}} Chapter 11.6.2), we conclude:
\begin{proposition}\label{invwingsprop}
Let $\beta\in (-1, 0)$. Then the linear operator $\Delta_{f}: \widetilde{W}^{k, p}_{\beta}(L)\to \widetilde{W}^{k-2, p}_{\beta}(L)$ is an isomorphism for $F:L\to\mathbb{C}^{m}$ being the Grim Reaper cylinder, that is, $F = F_{0}$ or $F = F^{\lambda}_{\phi}$. Furthermore,
there exists $C>0$ independent of $u$ such that
\begin{align}
    \|u\|_{\widetilde{W}^{k, p}_{\beta}}\leq C\cdot\|\Delta_{f}u\|_{\widetilde{W}^{k-2, p}_{\beta}}.
\end{align}
\end{proposition}

\subsection{Analysis on the conical singular ends}
Let $\mathbb{R}^{m}$, $m\geq 3$, be equipped with the Euclidean metric $g_{e}$. Let $\{p_{1}, \cdots, p_{l}\}\subset\mathbb{R}^{m}$ be a finite set of points. Then $\mathbb{R}^{m}\setminus \{p_{1}, \cdots, p_{l}\}$ can be viewed as a Riemannian manifold with $l$ CS ends $E_{j} := \overline{B_{\delta}(p_{j})}\setminus\{p_{j}\}$, with rate $\mu = 3$ and cone $C_{j} = (S^{m-1}\times\mathbb{R}, \:g_{C_{j}} = dr^{2} + r^{2}g_{S^{m-1}})$, $j = 1, \cdots, l$. We would like to study the mapping property of the weighted Laplacian $\Delta_{f}$ on the ends $E_{j}$.

To do this, we first study the mapping property  for the Laplacian $\Delta_{g}$ on Riemannian manifolds with CS ends, which should be well-known for the experts. As a result, most parts of this subsection are not new, as one can find more comprehensive treatments in this topic in {\cite{LM}}, {\cite{JoyceCS1}}, {\cite{Marshall}}, {\cite{PaciniUnifEst}}, to name a few.

Let $(L, g)$ be a Riemannian $m$-manifold with one CS end $E\subset L$. Then by definition there exists $\check{\epsilon}>0$, a compact Riemannian $(m-1)$-manifold $(\Sigma, g_{\Sigma})$, and a diffeomorphism $\check{\phi}: \Sigma\times(0, \check{\epsilon}]\to E$ such that
\begin{align*}
    |\nabla^{k}(\check{\phi}^{*}g - g_{C})|_{g_{C}} = O(r^{\mu-k})\quad\mbox{as $r\to 0$}
\end{align*}
for some $\mu\in\mathbb{R}$, where $g_{C} = dr^{2} + r^{2}g_{\Sigma}$. Let $\check{r}: L\to\mathbb{R}_{+}$ be a smooth function so that
$\check{r}\big|_{L\setminus E_{1}} = 1$ and $\check{r}\big|_{E}\circ\check{\phi}(r,\sigma) = r$, where $E^{1}:=\{\:x\in L\:|\:d_{g}(x, E)\leq 1\:\}$ is a neighbourhood of $E$ in $L$. Namely, $\check{r}$ is a smooth function that coincides with the coordinate function $r\in(0, \check{\epsilon}]$ on $E$. We will call $\check{r}$ a {\it radius function}.

Let $\Omega\subset L$ be an open domain containing $E$, with smooth boundary $\partial\Omega$. Define weighted $C^{k}$-spaces $C^{k}_{\mu}(\Omega)$ to be the space of $k$-times continuously differentiable functions on $\Omega$ so that the norm
\begin{align*}
    \|u\|_{C^{k}_{\nu}} := \sup_{\Omega}|\check{r}^{-\nu + j}\nabla^{j} u|_{g}
\end{align*}
is finite, where $|\cdot|$ and the connection $\nabla$ are computed using the CS metric $g$. The subspace $C^{k}_{\nu, \mathcal{D}}(\Omega)\subset C^{k}_{\nu}$ is defined to be the subspace of functions that vanish on $\partial\Omega$.

We also define the weighted Sobolev spaces $L^{k, p}_{\nu}(\Omega)$ to be the space of locally integrable $k$-times weakly differentiable functions such that the norm
\begin{align*}
    \|u\|_{L^{k, p}_{\nu}} := \left(\sum_{j=0}^{k}\int_{\Omega}|\check{r}^{-\nu+j}\nabla^{j}u|_{g}\:\check{r}^{-m}dV_{g}\right)^{\frac{1}{p}}
\end{align*}
is finite. Again, the norms $|\cdot|$ and connection $\nabla$ are computed using the metric $g$.The subspace $L^{k, p}_{\nu, \mathcal{D}}(\Omega)$ is defined to be the completion of $C^{k}_{\nu, \mathcal{D}}(\Omega)$ with respect to the $L^{k, p}_{\nu}$-norm above. The spaces $L^{k, p}_{\nu}(\Omega)$ and $L^{k, p}_{\nu, \mathcal{D}}(\Omega)$ are Banach spaces.

Recall the following classical result about the mapping property for Laplace operator $\Delta_{g}$.
\begin{proposition}\label{DirLap}
Let $m\geq 3$ and $\nu\in (2-m, 0)$. Then for any $k\geq 2$, $p>1$,
\begin{align}\label{IsomLap}
    \Delta_{g}: L^{k, p}_{\nu, \mathcal{D}}(\Omega)\to L^{k-2, p}_{\nu-2}(\Omega)
\end{align}
is an isomorphism.
\end{proposition}
\begin{proof}
It is well-known that there is no indicial roots of $\Delta_{g}$ in the interval $(2-m, 0)$, therefore, if $|u|\leq Cr^{\delta}$ for some $\delta>2-m$, then $u$ must be bounded as $r\to 0$. In particular, if $u\in L^{k, p}_{\nu, \mathcal{D}}(\Omega)$ for $\nu>2-m$ and $\Delta_{g}u = 0$, then $u$ must extend to be defined in the unweighted Sobolev space $L^{k, p}(\Omega)$ with zero Dirichlet boundary condition. By maximum principle, $u \equiv 0$, so $\Delta_{g}: L^{k, p}_{\nu, \mathcal{D}}(\Omega)\to L^{k-2, p}_{\nu-2}(\Omega)$ is injective for $\nu>2-m$.

For surjectivity, note that $\Delta_{g}$ is essentially self-adjoint and we have the following duality:
\begin{lemma}\label{IBP}
Let $p, q>1$ with $\frac{1}{p} + \frac{1}{q} = 1$, and $\nu\in\mathbb{R}$. Suppose $u\in L^{2, p}_{\nu, \mathcal{D}}(\Omega)$ and $v\in L^{2, q}_{-\nu + 2 - m}(\Omega)$. Then
\begin{align}
    \int_{\Omega}\Delta_{g}u\cdot v\:dV_{g} = -\int_{\Omega}\langle\nabla u, \nabla v\rangle_{g}\:dV_{g} = \int_{\Omega}u\cdot\Delta_{g}v\:dV_{g}.
\end{align}
\end{lemma}
It follows that $\Delta_{g}: L^{k, p}_{\nu, \mathcal{D}}(\Omega)\to L^{k-2, p}_{\nu-2}(\Omega)$ is surjective for $\nu<0$. Hence the isomorphism $(\ref{IsomLap})$ holds for $\nu\in (2-m, 0)$.

\end{proof}

Now let $F:L\to\mathbb{C}^{m}$ be a Lagrangian translating soliton with isolated conical singularities. For instance, $F = F_{0}$ or $F_{\phi}$ and $L = \mathbb{R}^{m}\setminus\{p_{1}, \cdots, p_{l}\}$. We prove an isomorphism theorem for $\Delta_{f} = \Delta_{g} - \frac{1}{2}\langle\nabla f, \nabla\cdot\rangle$, where $f(x) = 2\langle F(x), T\rangle$, by using a trick performed in the proof of \cite[Thoerem 5.3]{JoyceCS1} to transform $\Delta_{f}$ into a Laplace operator of some CS Riemannian metric, and then use Proposition $\ref{DirLap}$.

\begin{proposition}\label{invCS}
Let $m\geq 3$. Let $F:L\to\mathbb{C}^{m}$ be a Lagrangian translating soliton with isolated conical singularities $p_{1}, \cdots, p_{e}$ with cones $C_{j} = (\Sigma_{j}\times\mathbb{R}_{>0},\: g_{C_{j}})$, $j= 1, \cdots, e$, with induced metric $g$. By scaling, we may assume that the neighbourhoods $E^{1}_{j}$ of the CS ends $E_{j}$ of $L$ are disjoint. Then the radius function $\check{r}$ is well-defined on $L$. 

Let $\Omega_{i}\subset L$ be an open domain with smooth boundary such that $\Omega_{i}\cap E^{1}_{j} = \phi$ if $i\neq j$.
Then for any $k\geq 2$, $p>1$ and $\nu\in (2-m, 0)$,
\begin{align}
    \Delta_{f}: L^{k, p}_{\nu, \mathcal{D}}(\Omega_{i})\to L^{k-2, p}_{\nu-2}(\Omega_{i})
\end{align}
is an isomorphism.
\end{proposition}
\begin{proof}
Since $m\geq 3$, define a conformal transformation of the metric $g$ by $\widetilde{g}:=e^{\frac{f}{2-m}}g$. We claim that $\widetilde{g}$ is a CS metric asymptotic to $(C_{i},\:\widetilde{g}_{C_{i}} = e^{\frac{f(p_{i})}{2-m}}g_{C_{i}})$. Indeed, let $\check{\phi}_{i}:\Sigma_{i}\times (0, \check{\epsilon}]\to E_{i}$ be the diffeomorphism given by the definition of CS metric, then
\begin{align*}
    |\check{\phi}^{*}_{i}\widetilde{g} - \widetilde{g}_{C_{i}}|_{\widetilde{g}_{C_{i}}} &= |e^{\frac{f\circ\check{\phi}_{i}}{2-m}}\check{\phi}^{*}_{i}{g} - e^{\frac{f(p_{i})}{2-m}}g_{C_{i}}|_{\widetilde{g}_{C_{i}}}\\
    &\leq|e^{\frac{f\circ\check{\phi}_{i}}{2-m}}(\check{\phi}^{*}_{i}{g} - g_{C_{i}})|_{\widetilde{g}_{C_{i}}} + | (e^{\frac{f\circ\check{\phi}_{i}}{2-m}} - e^{\frac{f(p_{i})}{2-m}}) g_{C_{i}}|_{\widetilde{g}_{C_{i}}}\\
    &\leq O(r^{\mu}) + O(r^{\mu})\quad\mbox{as $r\to 0$},
\end{align*}
if $|\check{\phi}_{i}^{*}g - g_{C_{i}}|_{g_{C_{i}}} = O(r^{\mu})$ as $r\to 0$. The higher order estimates follow similarly. Hence the claim is proved by replacing $r$ by $e^{\frac{f(p_{i})}{4-2m}}r$.

Now it is straightforward to check that the Laplacian $\Delta_{\widetilde{g}}$ of the conformal metric $\widetilde{g}$ satisfies
\begin{align*}
    e^{\frac{f}{2-m}}\Delta_{\widetilde{g}} = \Delta_{f},
\end{align*}
and the weighted Sobolev spaces $L^{k, p}_{\nu}(\Omega_{i})$ defined using $g$ and $\widetilde{g}$ are equivalent. Since multiplication by $e^{\frac{f}{2-m}}$ gives an automorphism of $L^{k-2, p}_{\nu-2}(\Omega_{i})$, $\Delta_{f}: L^{k, p}_{\nu, \mathcal{D}}(\Omega_{i})\to L^{k-2, p}_{\nu-2}(\Omega_{i})$ is an isomorphism if and only if $\Delta_{\widetilde{g}}$ is. The proposition now follows from Proposition $\ref{DirLap}$.
\end{proof}

\subsection{Analysis on the wings}\label{invwings}
We now combine the results of previous subsections to prove isomorphism theorem for $\Delta_{f}$ on $\mathbb{R}^{m}\setminus\{p_{1}, \cdots, p_{e}\}$, equipped with Euclidean metric $g$. Here, we view the points $p_{i}$'s as isolated conical singularities with cone $(C_{i} = S^{m-1}\times\mathbb{R}_{+},\:g_{C_{i}} = g)$ and rate $\mu = 3$. By scaling the metric $g$, we may assume $d_{g}(p_{i}, p_{j})\geq 2$ and the radial function $\check{r}:\mathbb{R}^{m}\to\mathbb{R}_{+}$ is well-defined.

Consider the weighted Sobolev spaces $\widecheck{W}^{k, p}_{\beta, \nu}(\mathbb{R}^{m})$ of locally integrable $k$-times weakly differentiable functions so that the norm
\begin{align*}
    \|u\|_{\widecheck{W}^{k, p}_{\beta, \nu}} := \left(\sum_{j=0}^{k}\int_{\mathbb{R}^{m}}|e^{\frac{\beta f}{2}}\check{r}^{-\nu + j}\nabla^{j}u|^{p}\:\check{r}^{-m}dV_{g_{0}}\right)^{\frac{1}{p}}
\end{align*}
is finite. Note that since $e^{\frac{\beta f}{2}}$ is bounded near the punctures $p_{i}$, the $\widecheck{W}^{k, p}_{\beta, \nu}$-norm is equivalent to the $L^{k, p}_{\nu}$-norm near the CS ends, and since $\check{r} = 1$ away from the punctures, the $\widecheck{W}^{k, p}_{\beta, \nu}$-norm is equivalent to the $\widetilde{W}^{k, p}_{\beta}$-norm away from the CS ends.

We can now prove the invertibility of $\Delta_{f}$.
\begin{proposition}\label{isowings}
Suppose $m\geq 3$. Let $\beta\in (-1, 0)$ and $\nu\in (2-m, 0)$. Then for any $k\geq 2$, $p>1$,
\begin{align}
    \Delta_{f}: \widecheck{W}^{k, p}_{\beta, \nu}(\mathbb{R}^{m})\to \widecheck{W}^{k-2, p}_{\beta, \nu-2}(\mathbb{R}^{m})
\end{align}
is an isomorphism.
\end{proposition}
\begin{proof}
First we show injectivity. Let $u\in \widecheck{W}^{k, p}_{\beta, \nu}(L)$ with $\Delta_{f}u = 0$. We claim that $u$ actually belongs to $\widetilde{W}^{k, p}_{\beta}$ if $\nu>2-m$. In fact, as in the proof of Proposition $\ref{invCS}$, $\Delta_{\widetilde{g}}u = e^{\frac{f}{m-2}}\Delta_{f}u = 0$, where $\widetilde{g} = e^{\frac{f}{2-m}}g$ is a CS Riemannian metric. Now the proof follows the same lines of the injectivity part of Proposition $\ref{DirLap}$. Since $\nu>2-m$, $u$ is actually bounded on the CS ends, and therefore $u$ belongs to $\widetilde{W}^{k, p}_{\beta}$. From Proposition $\ref{invwingsprop}$, $u = 0$, so $\Delta_{f}$ is injective.

Next we prove surjectivity. Given $h\in \widecheck{W}^{k-2, p}_{\beta,\nu-2}(L)$.
For each CS end $E_{i}$ of $L$, let $\Omega_{i}\supset E_{i}$ be a smooth open subset which does not intersect $E_{j}$ if $i\neq j$. Let $\varphi_{i}$ be the smooth cut-off function satisfies $\varphi_{i}\big|_{E_{i}}\equiv 1$ and $\varphi_{i}\big|_{\mathbb{R}^{m}\setminus\Omega_{i}}\equiv 0$. Then 
\begin{align*}
    h = \sum_{i=1}^{e}\varphi_{i}h + \left(1-\sum_{i=1}^{e}\varphi_{i}\right)h.
\end{align*}
By Proposition $\ref{invCS}$, for each $i = 1, \cdots, e$, there exists $u_{i}\in \widecheck{W}^{k, p}_{\nu, \mathcal{D}}(\Omega_{i})$ solving  $\Delta_{f}u_{i} = \varphi_{i} h$, and by proposition $\ref{invwingsprop}$, there exists $\widetilde{u}\in \widetilde{W}^{k, p}_{\beta}(L)$ so that $\Delta_{f}\widetilde{u} = (1-\sum_{i=1}^{e}\varphi_{i})h$. Finally, set $u = \sum_{i=1}^{e}u_{i} + \widetilde{u}$, then it is easy to see that $u\in \widecheck{W}^{k, p}_{\beta, \nu}(L)$ and $\Delta_{f}u = h$. The proof is completed.

\end{proof}

\subsection{Analysis on the Lawlor necks}\label{invneck}
Consider the special Lagrangian Lawlor neck $N: S^{m-1}\times\mathbb{R}\to\mathbb{C}^{m}$, $m\geq 3$,  defined in \S3.2 with phase $\overline{\theta}$, equipped  with induced AC metric $g$. Then $N$ is not an $f$-special Lagrangian, indeed, the  Lagrangian angle $\theta_{f}$ with respect to $\Omega_{f}$ is given by $\theta_{f} = \overline{\theta} - \langle N, JT\rangle$. Without loss of generality, we may assume $\overline{\theta} = 0$.

Since $N$ is AC with cone $C = \Pi_{0}\cup \Pi_{1}$ and rate $\lambda = 2-m$, there exists a compact subset 
$K\subset S^{m-1}\times\mathbb{R}$, a constant $\hat{R}>0$ and a diffeomorphism $\hat{\phi}:(S^{m-1}\cup S^{m-1})\times[\hat{R}, \infty)$ so that
\begin{align*}
    |\nabla^{k}(N\circ\hat{\phi} - \iota_{C})|_{g_{C}} = O(r^{1-m-k})\quad\mbox{as $r\to\infty$}.
\end{align*}
Define a {\it radius functions} $\hat{r}: S^{m-1}\times\mathbb{R}\to\mathbb{R}_{+}$ to be a smooth function such that
\begin{align*}
    \hat{r}\big|_{(S^{m-1}\times\mathbb{R})\setminus K}\circ\hat{\phi}(\sigma, r) = r,
\end{align*}
that is, $\hat{r}$ coincides with the coordinate function $r$ on the AC ends of $S^{m-1}\times\mathbb{R}$, under the indentification of $\hat{\phi}$.

Since the size of the neck region in the approximate solution $X_{t}$ is dependent on a parameter $t\in(0, \delta)$, we will consider the weighted Sobolev spaces of the rescaled Lawlor neck $N_{t} := tN$. The induced metric is $g_{t}:=t^{2}g$ and the rescaled radius function is $\hat{r}_{t}:=t\hat{r}$.
\begin{definition}
For $k\geq 0$, $p>1$, the weighted Sobolev space $\widehat{W}^{k, p}_{\gamma, t}(S^{m-1}\times\mathbb{R})$ is defined to be the completion of $C^{\infty}_{c}(S^{m-1}\times\mathbb{R})$ with respect to the norm
\begin{align*}
\|u\|_{\widehat{W}^{k, p}_{\gamma, t}} := \left(\sum_{j=1}^{k}\int_{S^{m-1}\times\mathbb{R}}|\hat{r}_{t}^{-\gamma +j}\nabla^{j}u|^{p}_{g_{t}}\:\hat{r}_{t}^{-m}dV_{g_{t}}\right)^{\frac{1}{p}}.
\end{align*}
\end{definition}

For $t$ sufficiently small, we shall study the invertibility of the linear operator
\begin{align*}
\mathcal{L}_{g_{t}} = e^{-\frac{f_{t}}{2}}(\:\cos\theta_{f_{t}}\:\Delta_{f_{t}} - \sin\theta_{f_{t}}\langle\nabla\theta_{f_{t}}, \nabla\cdot\rangle_{g_{t}})
\end{align*}
on the scaled Lawlor neck $N_{t}:S^{m-1}\times\mathbb{R}\to\mathbb{C}^{m}$ with the induced AC metric $g_{t}$, where $f_{t} =t\cdot f$, $\Delta_{f_{t}} = \Delta_{g_{t}} - \frac{1}{2}\langle\nabla f_{t}, \nabla\cdot\rangle_{g_{t}}$, and $\theta_{f_{t}} = - \langle N_{t}, JT\rangle$. 
We first recall the mapping property of the Laplace operator $\Delta_{g_{t}}$ on AC Riemannian manifolds.
\begin{proposition}\label{lapinv}
Given any $t>0$, $k\geq 2$, $p>1$, $\gamma\in (2-m, 0)$,
the Laplace operator $\Delta_{g_{t}}: \widehat{W}^{k, p}_{\gamma, t}(S^{m-1}\times\mathbb{R})\to \widehat{W}^{k-2, p}_{\gamma-2, t}(S^{m-1}\times\mathbb{R})$ is an isomorphism. Furthermore, there is a constant $C>0$ independent of $t$ such that
\begin{align}\label{invLap}
\|u\|_{\widehat{W}^{k, p}_{\gamma, t}}\leq C\|\Delta_{g_{t}}u\|_{\widehat{W}^{k-2, p}_{\gamma-2, t}}.
\end{align}
\end{proposition}
\begin{proof}
It is well-known that the Proposition holds for each $t$, hence it remains to show that the constant $C$ is independent of $t$. Observe that the $\widehat{W}^{k, p}_{\gamma, t}$-norm and the Laplacian scale like 
\begin{align*}
    \|u\|_{\widehat{W}^{k, p}_{\gamma, t}} = t^{-\gamma}\|u\|_{\widehat{W}^{k, p}_{\gamma, 1}},\quad\mbox{and} \quad\Delta_{g_{t}}u = t^{-2}\Delta_{g}u.
\end{align*}
Hence, if $C_{1}>0$ is the constant for $t = 1$,
\begin{align*}
    \|u\|_{\widehat{W}^{k, p}_{\gamma, t}} = t^{-\gamma}\|u\|_{\widehat{W}^{k, p}_{\gamma, 1}}\leq t^{-\gamma}C_{1}\|\Delta_{g}u\|_{\widehat{W}^{k-2, p}_{\gamma-2, 1}} = t^{-\gamma}C_{1}t^{\gamma-2}t^{2}\|\Delta_{g_{t}}u\|_{\widehat{W}^{k-2, p}_{\gamma-2, t}} = C_{1}\|\Delta_{g_{t}}u\|_{\widehat{W}^{k-2, p}_{\gamma-2, t}}.
\end{align*}
\end{proof}

Set $\mathcal{P}_{t}:=e^{\frac{f_{t}}{2}}\mathcal{L}_{g_{t}} = \cos\theta_{f_{t}}\:\Delta_{f_{t}} - \sin\theta_{f_{t}}\langle\nabla\theta_{f_{t}}, \nabla\cdot\rangle_{g_{t}}$. Unlike the Laplacian $\Delta_{g_{t}}$, since $\nabla f_{t}$ does not decay as $r\to\infty$, $\mathcal{P}_{t}$ does not map $\widehat{W}^{k, p}_{\gamma, t}$ to $\widehat{W}^{k-2, p}_{\gamma-2, t}$. However, we can still get uniform estimate for functions in $\widehat{W}^{k, p}_{\gamma, t}$ with compact support, provided $t$ is small enough. Let $b\in (0, 1)$. For $t$ small enough such that $t\hat{R}<t^{b}$, consider the domain 
\begin{align*}
    \Omega_{t}:=\{\:x\in S^{m-1}\times\mathbb{R}\:|\:\hat{r}_{t}(x)<t^{b}\}
\end{align*}
We now prove the uniform estimate of $\mathcal{P}_{t}$ for small $t$ in $\Omega_{t}$.
\begin{proposition}\label{estneck}
There exists $\delta>0$ such that for any $t\in (0, \delta)$, $k\geq 2$, $p>1$, $\gamma\in (2-m, 0)$, there exists $C>0$ independent of $t$ so that for any $u\in \widehat{W}^{k, p}_{\gamma, t}(S^{m-1}\times\mathbb{R})$ supported in $\Omega_{t}$,
\begin{align}\label{NeckSchauder}
    \|u\|_{\widehat{W}^{k, p}_{\gamma, t}}\leq C\|\mathcal{P}_{t}u\|_{\widehat{W}^{k-2, p}_{\gamma-2, t}}.
\end{align}
As a result, for any $t\in (0, \delta)$, given $v\in \widehat{W}^{k-2, p}_{\gamma-2, t}(\Omega_{t})$, there exists a unique $u\in\widehat{W}^{k, p}_{\gamma, t, \mathcal{D}}(\Omega_{t})$ such that $\mathcal{P}_{t}u = v$ and $\|u\|_{\widehat{W}^{k, p}_{\gamma, t}}\leq C\|v\|_{\widehat{W}^{k-2, p}_{\gamma-2, t}}$, where $\widehat{W}^{k, p}_{\gamma, t, \mathcal{D}}(\Omega_{t})$ denotes the functions in $\widehat{W}^{k, p}_{\gamma, t}(\Omega_{t})$ with zero Dirichlet boundary condition on $\partial\Omega_{t}$.
\end{proposition}
\begin{proof}
Since $N_{t}$ is AC, it is not hard to see that $|N_{t}| = O(\hat{r}_{t})$. Hence we may 
choose $\delta>0$ small enough so that $\sup_{\Omega_{t}}|\theta_{f_{t}}| < \frac{\pi}{3}$ for $t\in (0, \delta)$. Then by Proposition $\ref{lapinv}$, there exists $C>0$ such that
\begin{align*}
    \|u\|_{\widehat{W}^{k, p}_{\gamma, t}}\leq C\|\Delta_{g_{t}}u\|_{\widehat{W}^{k-2, p}_{\gamma-2, t}}\leq 2C\|\cos\theta_{f_{t}}\Delta_{g_{t}}u\|_{\widehat{W}^{k-2, p}_{\gamma-2, t}} = 2C\|\mathcal{P}_{t}u + \mathcal{J}_{t}u\|_{\widehat{W}^{k-2, p}_{\gamma-2, t}},
\end{align*}
where $\mathcal{J}_{t}$ is a first order differential operator defined by 
\begin{align*}
    \mathcal{J}_{t}u := \frac{1}{2}\cos\theta_{f_{t}}\langle\nabla f_{t}, \nabla u\rangle + \sin\theta_{f_{t}}\langle\nabla\theta_{f_{t}}, \nabla u\rangle.
\end{align*}
Let $V_{t}:=\frac{1}{2}\cos\theta_{f_{t}}\nabla f_{t} + \sin\theta_{f_{t}}\nabla\theta_{f_{t}}$, then $|\nabla^{k}V_{t}|_{g_{t}}$ are bounded uniformly in $t$, for all $k\geq 0$. Since $\supp u\subset \Omega_{t}$,
\begin{align}
    \|\mathcal{J}_{t}u\|^{p}_{\widehat{W}^{k-2, p}_{\gamma-2, t}} = \|\langle V_{t}, \nabla u\rangle_{g_{t}}\|^{p}_{\widehat{W}^{k-2, p}_{\gamma-2, t}} &= \sum_{j=0}^{k-2}\int_{\Omega_{t}}|\:\hat{r}_{t}^{-\gamma + 2 + j}\nabla^{j}\langle V_{t}, \nabla u\rangle_{g_{t}}\:|^{p}_{g_{t}}\:\hat{r}_{t}^{-m}dV_{g_{t}}\nonumber\\
    &\leq\sum_{j=0}^{k-2}\sum_{l=0}^{j}\int_{\Omega_{t}}|\:\hat{r}^{-\gamma + 2 + j}_{t}\nabla^{l+1}u\:|^{p}_{g_{t}}\:|\nabla^{j-l}V_{t}|^{p}_{g_{t}}\:\hat{r}^{-m}_{t}dV_{g_{t}}\nonumber\\
    &\leq C_{1}\sum_{j=0}^{k}\int_{\Omega_{t}}|\:\hat{r}^{-\gamma + 2 + j}_{t}\nabla^{j}u\:|^{p}_{g_{t}}\:\hat{r}^{-m}_{t}dV_{g_{t}}\nonumber\\
    &\leq C_{1}\cdot t^{2bp}\sum_{j=0}^{k}\int_{\Omega_{t}}|\:\hat{r}^{-\gamma + j}_{t}\nabla^{j}u\:|^{p}_{g_{t}}\:\hat{r}^{-m}_{t}dV_{g_{t}}  = C_{1}\cdot t^{2bp}\:\|u\|^{p}_{\widehat{W}^{k, p}_{\gamma, t}}\label{firstorder}
\end{align}
for some $C_{1}>0$ independent of $t$, where we have used $\hat{r}_{t}^{2}(x) < t^{2b}$ for $x\in \Omega_{t}$ in the last inequality. Choose $\delta>0$ even smaller so that $2C\cdot C_{1}^{1/p}t^{2b}<\frac{1}{2}$ for $t\in (0, \delta)$, we have
\begin{align*}
    \|u\|_{\widehat{W}^{k, p}_{\gamma, t}}\leq 2C\|\mathcal{P}_{t}\|_{\widehat{W}^{k-2, p}_{\gamma-2, t}} + \frac{1}{2}\|u\|_{\widehat{W}^{k, p}_{\gamma, t}}
\end{align*}
for $t\in (0, \delta)$. Hence $(\ref{NeckSchauder})$ is proved. The last assertion now follows from applying method of continuity to the family of operators
\begin{align*}
    \mathcal{P}^{s}_{t}:=\cos(\theta_{sf_{t}})\Delta_{sf_{t}} - \sin(\theta_{sf_{t}})\langle\nabla\theta_{sf_{t}}, \nabla\cdot\rangle,\quad s\in[0, 1],
\end{align*}
for each $t\in(0, \delta)$.
\end{proof}

\subsection{Uniform estimate for the linearized operator on the approximate solutions}
Now we are ready to prove the uniform invertibility of $\mathcal{L}_{g_{t}}$ on the weighted Sobolev space $W^{k, p}_{\beta, \gamma, t}(X_{t})$ which has been defined in \S4.2. Observe that for functions supported on the wings the $W^{k, p}_{\beta, \gamma, t}$-norm is equivalent to the $\widecheck{W}^{k, p}_{\beta, \gamma}$-norm, while for functions supported on the neck region, the $W^{k, p}_{\beta, \gamma, t}$-norm is equivalent to $\widehat{W}^{k, p}_{\gamma, t}$-norm. In the following we define a cut-off function as in \cite[Theorem 12.2]{PaciniUnifEst} to decompose functions on $X_{t}$ into functions supported on the wings and the neck regions, and with uniform estimate for all small $t$.

For any $\tau>0$, choose $a, b\in\mathbb{R}$ such that $0<b<a<\tau$. Let $\eta: \mathbb{R}\to[0,1]$ be a smooth decreasing function such that $\eta(s)=1$ for $s\leq b$ and $\eta(s)=0$ for $s\geq a$. For $t\in(0, \delta)$, define
\begin{align*}
\eta_{t}: (0, \infty)\to[0,1],\quad\eta_{t}(r):=\eta\left(\frac{\log r}{\log t}\right).
\end{align*}
Then $\eta_{t}(r) = 1$ for $r\geq t^{b}$, $\eta_{t}(r) = 0$ for $r\leq t^{a}$, and for each $k\in\mathbb{N}$, there exists $C_{k}>0$ such that
\begin{align}\label{eta}
\left|\:r^{-k}\frac{\partial^{k}\eta_{t}}{\partial r^{k}}(r)\right|\leq\frac{C_{k}}{|\log t|}.
\end{align}
We may extend $\eta_{t}$ to a globally defined function on $X_{t}$ in an obvious manner. Notice that
by choosing $\delta$ even smaller if necessary (specifically, $\delta<(\frac{1}{2})^{\tau - a}$), for $t\in (0, \delta)$ we have the following ordering:
\begin{align*}
    0<t\hat{R}<t^{\tau}<2t^{\tau}<t^{a}<t^{b}<\epsilon.
\end{align*}

By using $\eta_{t}$ as cut-off functions to decompose $u$ into corresponding parts supported in the wings and the necks, we may use the results proved in Section $\ref{invwings}$ and $\ref{invneck}$ and patch them together. The estimate $(\ref{eta})$ ensures that the error produced by this process can be made arbitrarily small. Again, as in \S4.1, we will only present the proof for the case that $X_{t}$ has only one neck region. The proof for general cases is essentially the same, since we have proven the invertibility of $\Delta_{f}$ on $\mathbb{R}^{m}\setminus\{p_{1}, \cdots, p_{e}\}$.

We first consider the operator $\mathcal{P}_{t}:=e^{\frac{f_{t}}{2}}\mathcal{L}_{g_{t}}$.
\begin{proposition}\label{InvL}
Let $k\geq 2$, $p>1$, $\beta\in (-1, 0)$ and $\gamma\in (2-m, 0)$. For $\delta>0$ small enough, there exists $C>0$ independent of $t\in (0, \delta)$ such that for any $u\in W^{k, p}_{\beta, \gamma, t}(X_{t})$,
\begin{align*}
\|u\|_{W^{k, p}_{\beta, \gamma, t}}\leq C\:\|\mathcal{P}_{t} u\|_{W^{k-2,p}_{\beta, \gamma-2, t}}.
\end{align*}
\end{proposition}

\begin{proof}
For any $u\in W^{k, p}_{\beta, \gamma, t}(X_{t})$, write $u = \eta_{t}u + (1-\eta_{t})u$. Then $\eta_{t}u$ is supported on $[(\mathbb{R}^{m}_{0}\cup\mathbb{R}^{m}_{\phi})\setminus E_{p}]\cup (\Sigma\times[t^{a}, \epsilon])$, and $(1-\eta_{t})u$ is supported on $(\Sigma\times[t\hat{R}, t^{b}])\cup [(S^{m-1}\times\mathbb{R})\setminus \hat{E}_{\infty}]$, where $\Sigma = S^{m-1}\cup S^{m-1}$.

Since each connected component of $[(\mathbb{R}^{m}_{0}\cup\mathbb{R}^{m}_{\phi})\setminus E_{p}]\cup (\Sigma\times[t^{a}, \epsilon])$ is identified with $\mathbb{R}^{m}\setminus B_{t^{a}}(0)$ with Euclidean metric and $\mathcal{P}_{t} = \Delta_{f}$, by Proposition $\ref{isowings}$  we have
\begin{align*}
\|\eta_{t}u\|_{W^{k, p}_{\beta, \gamma, t}} = \|\eta_{t} u\|_{\widecheck{W}^{k, p}_{\beta, \gamma}}\leq C\| \Delta_{f}(\eta_{t}u)\|_{\widecheck{W}^{k-2, p}_{\beta, \gamma-2}}
\end{align*}
for some $C>0$. Then 
\begin{align*}
    \| \Delta_{f}(\eta_{t}u)\|_{\widecheck{W}^{k-2, p}_{\beta, \gamma-2}} = \left[\;\sum_{j=0}^{k-2}\int_{\mathbb{R}^{m}}|e^{\frac{\beta f}{2}}\check{r}^{-\gamma+2+j}\nabla^{j}(\Delta_{f}(\eta_{t}u))|^{p}_{g_{0}}\:\check{r}^{-m}dV_{g_{0}}\;\right]^{\frac{1}{p}}.
\end{align*}
By direct computation we have
\begin{align}\label{expand}
    \Delta_{f}(\eta_{t}u) = \Delta_{f}\eta_{t}\cdot u + \eta_{t}\cdot\Delta_{f}u + 2\langle\nabla\eta_{t}, \nabla u\rangle_{g_{0}}.
\end{align}
The first term of ($\ref{expand}$) is estimated by:
\begin{align*}
    &\sum_{j=0}^{k-2}\int_{\mathbb{R}^{m}}|e^{\frac{\beta f}{2}}\check{r}^{-\gamma+2+j}\nabla^{j}(\Delta_{f}\eta_{t}\cdot u)|^{p}_{g_{0}}\:\check{r}^{-m}dV_{g_{0}}\\
    &\leq\sum_{j=0}^{k-2}\sum_{l=0}^{j}\int_{\mathbb{R}^{m}}|e^{\frac{\beta f}{2}}\check{r}^{-\gamma+2+j}\nabla^{j-l}\Delta_{f}\eta_{t}\cdot \nabla^{l}u|^{p}_{g_{0}}\:\check{r}^{-m}dV_{g_{0}}\\
    &\leq C\sum_{j=0}^{k-2}\sum_{l=0}^{j}\int_{\mathbb{R}^{m}}|e^{\frac{\beta f}{2}}\check{r}^{-\gamma+2+j}\nabla^{j-l+2}\eta_{t}\cdot \nabla^{l}u|^{p}_{g_{0}}\:\check{r}^{-m}dV_{g_{0}} \\
    &\hspace{5cm}+ C\sum_{j=0}^{k-2}\sum_{l=0}^{j}\int_{\mathbb{R}^{m}}|e^{\frac{\beta f}{2}}\check{r}^{-\gamma+2+j}\nabla^{j-l+1}\eta_{t}\cdot \nabla^{l}u|^{p}_{g_{0}}\:\check{r}^{-m}dV_{g_{0}}\\
    &=C\sum_{j=0}^{k-2}\sum_{l=0}^{j}\int_{\mathbb{R}^{m}}|e^{\frac{\beta f}{2}}\check{r}^{-\gamma+l}(\check{r}^{j-l+2}\nabla^{j-l+2}\eta_{t})\cdot \nabla^{l}u|^{p}_{g_{0}}\:\check{r}^{-m}dV_{g_{0}} \\
    &\hspace{5cm}+ C\sum_{j=0}^{k-2}\sum_{l=0}^{j}\int_{\mathbb{R}^{m}}|e^{\frac{\beta f}{2}}\check{r}^{-\gamma+1+l}(\check{r}^{j-l+1}\nabla^{j-l+1}\eta_{t})\cdot \nabla^{l}u|^{p}_{g_{0}}\:\check{r}^{-m}dV_{g_{0}}\\
    &\leq \frac{C}{|\log t|^{p}}\sum_{j=0}^{k-2}\sum_{l=0}^{j}\int_{\mathbb{R}^{m}}|e^{\frac{\beta f}{2}}\check{r}^{-\gamma+l} \nabla^{l}u|^{p}_{g_{0}}\:\check{r}^{-m}dV_{g_{0}}\leq \frac{C}{|\log t|^{p}}\|u\|^{p}_{W^{k, p}_{\beta, \gamma}}.
\end{align*}
Note that we have used ($\ref{eta}$) and $\check{r}\leq 1$ in the third inequality above. By similar computation we can obtain the estimates
\begin{align*}
    &\sum_{j=0}^{k-2}\int_{\mathbb{R}^{m}}|e^{\frac{\beta f}{2}}\check{r}^{-\gamma+2+j}\nabla^{j}(\eta_{t}\cdot \Delta_{f}u)|^{p}_{g_{0}}\:\check{r}^{-m}dV_{g_{0}}\leq C\left( 1 + \frac{1}{|\log t|}\right)^{p}\|\Delta_{f}u\|^{p}_{W^{k-2, p}_{\beta, \gamma-2}},\\
    &\sum_{j=0}^{k-2}\int_{\mathbb{R}^{m}}|e^{\frac{\beta f}{2}}\check{r}^{-\gamma+2+j}\nabla^{j}(\langle\nabla\eta_{t}, \nabla u\rangle)|^{p}_{g_{0}}\:\check{r}^{-m}dV_{g_{0}}\leq\frac{C}{|\log t|^{p}}\|u\|^{p}_{W^{k, p}_{\beta, \gamma}}.
\end{align*}
Therefore
\begin{align}\label{www}
    \|\eta_{t} u\|_{W^{k, p}_{\beta, \gamma, t}} \leq C\|\mathcal{P}_{t}u\|_{W^{k-2, p}_{\beta, \gamma-2, t}} + \frac{C}{|\log t|}\|u\|_{W^{k, p}_{\beta, \gamma, t}}.
\end{align}
Next, we estimate the norm of $(1-\eta_{t})u$, which is supported on $(\Sigma\times[t\hat{R}, t^{b}])\cup [(S^{m-1}\times\mathbb{R})\setminus \hat{E}_{\infty}]$. Under the identification of $\hat{\phi}$, we may think of $(1-\eta_{t})u$ as a function on $S^{m-1}\times\mathbb{R}$, with induced metric equivalent to the AC metric $g_{t}$ on the Lawlor neck by $(\ref{metric})$. By Proposition $\ref{estneck}$ we have
\begin{align*}
    \|(1-\eta_{t})u\|_{{W}^{k, p}_{\beta, \gamma, t}}\leq C\|(1-\eta_{t})u\|_{\widehat{W}^{k, p}_{\gamma, t}}\leq C\|\mathcal{P}_{t}[(1-\eta_{t})u]\|_{\widehat{W}^{k-2, p}_{\gamma-2, t}}.
\end{align*}
By similar computation as that for $\eta_{t}u$, we obtain
\begin{align}\label{nnn}
    \|(1-\eta_{t})u\|_{W^{k, p}_{\beta, \gamma, t}}\leq C\|\mathcal{P}_{t}u\|_{W^{k-2, p}_{\beta, \gamma-2, t}} + \frac{C}{|\log t|}\|u\|_{W^{k, p}_{\beta, \gamma, t}}.
\end{align}
Combining ($\ref{www}$) and ($\ref{nnn}$), we find 
\begin{align*}
    \|u\|_{W^{k, p}_{\beta, \gamma, t}}\leq \|\eta_{t}u\|_{W^{k, p}_{\beta, \gamma, t }} + \|(1-\eta_{t})u\|_{W^{k, p}_{\beta, \gamma, t}}\leq C\|\mathcal{P}_{t}u\|_{W^{k-2, p}_{\beta, \gamma-2, t}} + \frac{C}{|\log t|}\|u\|_{W^{k, p}_{\beta, \gamma, t}}
\end{align*}
for some $C>0$ independent of $t$. Setting $\delta>0$ small enough so that $\frac{C}{|\log t|}<\frac{1}{2}$, then the result follows.

\end{proof}

We finally obtain the uniform invertibility of $\mathcal{L}_{g_{t}}$ by composing the isomorphism $u\mapsto e^{-\frac{f_{t}}{2}}u$ with the isomorphism $\mathcal{P}_{t}$.
\begin{proposition}\label{LinearIsom}
For $k\geq 2$, $p>1$, $\beta\in (-1, 0)$ and $\gamma\in (2-m, 0)$, there exists $\delta>0$ so that the bounded operator $\mathcal{L}_{g_{t}}: W^{k, p}_{\beta, \gamma, t}(X_{t})\to W^{k-2, p}_{\beta+1, \gamma-2, t}(X_{t})$ is an isomorphism for all $t\in (0, \delta)$. Moreover, there exists $C>0$ independent of $t$ such that
\begin{align}\label{UnifInv}
    \|u\|_{W^{k, p}_{\beta, \gamma, t}}\leq C\:\|\mathcal{L}_{g_{t}} u\|_{W^{k-2,p}_{\beta+1, \gamma-2, t}}.
\end{align}
\end{proposition}

\section{Perturbation}

\subsection{Lagrangian Neighbourhood Theorem}
Given a Lagrangian submanifold $F:L\to\mathbb{C}^{m}$, by a {\it Lagrangian neighobourhood} we mean a pair $(U_{L}, \Phi_{L})$, where $U_{L}\subset T^{*}L$ is an open neighbourhood of the zero section $\underline{0}\in\Gamma(T^{*}L)$, and $\Phi_{L}:U_{L}\to\mathbb{C}^{m}$ is an embedding such that $\Phi_{L}^{*}\omega_{0} = \hat{\omega}$ and $\Phi_{L}\big|_{\underline{0}} = F$, where $\hat{\omega}$ is the canonical symplectic form on $T^{*}L$. Here, we have identified the zero section $\underline{0}$ with the manifold $L$ itself. We shall construct a Lagrangian neighbourhood of the approximate solution  $\widetilde{F}_{t}:X_{t}\to\mathbb{C}^{m}$ by combining the Lagrangian neighbourhoods of each part of $X_{t}$.

\

\noindent{\it Lagrangian neighbourhood of Grim Reaper cylinder.}\quad
Let $F = F_{0}:\mathbb{R}^{m}\to\mathbb{C}^{m}$ be a Grim Reaper cylinder. We will construct a Lagrangian neighbourhood for $F$. The Lagrangian neoghbourhood of $F^{\lambda}_{\phi}$ can be obtained by composing the Lagrangian meighbourhood of $F_{0}$ with the affine transformation $P_{(\phi, \lambda)}$. 

We need the following theorem.
\begin{theorem}[Weinstein \cite{Weinstein}]\label{weinstein} Let $(M^{2m}, \omega)$ be a symplectic manifold and $i_{L}:L\to M$ be an $m$-dimensional submanifold. Suppose $\{P_{x}\:|\:x\in L\}$ is a smooth family of embedded, noncompact, Lagrangian submanifolds such that $x\in P_{x}$ and $T_{x}L\cap T_{x}P_{x} = \{0\}$. Then there exists an open neighbourhood $U_{L}\subset T^{*}L$ containing the zero section $\underline{0}$ suhc that the fibers of the natural projection $\pi:U_{L}\to L$ is connected, and there is a unique embedding $\Phi_{L}:U_{L}\to M$ with $\Phi_{L}(\pi^{-1}(x))\subset P_{x}$, $\Phi_{L}\big|_{\underline{0}} = i_{L}$, and $\Phi_{L}^{*}\omega = \hat{\omega} + \pi^{*}(i_{L}^{*}\omega)$.
\end{theorem}
In particular, when $i_{L}:L\to M$ is Lagrangian, then $(U_{L}, \Phi_{L})$ defines a Lagrangian neighbourhood of $L$ in $M$. Thus, we have to find the family $\{P_{x}\:|\:x\in L\}$ for $L = F(\mathbb{R}^{m})$ the Grim Reaper cylinder.

Let $\gamma:\mathbb{R}\to\mathbb{C}$ be the Grim Reaper curve, parametrized by arc-length, such that $F(x_{1}, \cdots, x_{m}) = (x_{1}, \cdots, x_{m-1}, \gamma(x_{m}))$. Let $\nu_{m}(x_{m}) := i\gamma'(x_{m})$ and let $\{\nu_{j}:= Je_{j}\}_{j=1}^{m-1}$, where $\{e_{j}\}_{j=1}^{m-1}$ is  the standard basis of $\mathbb{R}^{m-1}\subset\mathbb{C}^{m-1}$, and $J$ is the standard complex structure of $\mathbb{C}^{m-1}$. Then $\{\nu_{j}\}_{j=1}^{m}$ is an orthonormal frame of the normal bundle $T^{\perp}L$. Let $x=F(x_{1}, \cdots, x_{m})\in\mathbb{C}^{m}$, define
\begin{align*}
    P_{x} = \left\{\sum_{j=1}^{m}\xi_{j}\nu_{j}(x)\::\:\sum_{j=1}^{m}\xi_{j}^{2}<d^{2}\right\}
\end{align*}
for $d>0$. Then for each $x = F(x_{1}, \cdots, x_{m})$, $P_{x}$ is Lagrangian and satisfies the required property of Theorem $\ref{weinstein}$. Hence we have a Lagrangian neighbourhood $(U, \Phi)$ to the Grim Reaper cylinder $F:\mathbb{R}^{m}\to\mathbb{C}^{m}$. We will denote the respecting Lagrangian neighbourhoods of the Grim Reaper cylinders $F_{0}$ and $F^{\lambda}_{\phi}$ by $(U_{0}, \Phi_{0})$ and $(U^{\lambda}_{\phi}, \Phi^{\lambda}_{\phi})$.

\

\noindent{\it Lagrangian neighbourhood of Lawlor necks.}\quad
Following the construction of Pacini {\cite{PaciniGlue}}, for any $t\in(0, \delta)$, we have a Lagrangian neighbourhood $(U_{t, N}, \Phi_{t, N})$ for the AC, Lagrangian submanifold $tN: S^{m-1}\times\mathbb{R}\to\mathbb{C}^{m}$. Indeed, if we have a Lagrangian neighbourhood $(U_{N}, \Phi_{N})$, then for $t\in (0, \delta)$, define an $\mathbb{R}^{+}$-action on $T^{*}N$ by $t\cdot(x, \alpha) = (x, t^{2}\alpha)$, and define
\begin{align*}
U_{t, N}:=t\cdot U_{N},\quad\Phi_{t, N}:=t\Phi_{N}t^{-1}.
\end{align*}
Then $(U_{t, N}, \Phi_{t, N})$ is the desired Lagrangian neighbourhood for $tN$.

\

\noindent{\it Lagrangian neighbourhood of the approximate solutions.}\quad
We can now define the Lagrangian neighbourhood for the approximate solution $\widetilde{F}_{t}: X_{t}\to\mathbb{C}^{m}$. Recall that we have a decomposition
\begin{align*}
X_{t} := [\:(\mathbb{R}^{m}_{0}\cup\mathbb{R}^{m}_{\phi})\setminus E_{p}\:]\cup [\: \sigma\times [t\hat{R}, \epsilon]\:]\cup [\:(S^{m-1}\times\mathbb{R})\setminus \hat{E}_{\infty}\:],
\end{align*}
where $(\mathbb{R}^{m}_{0}\cup\mathbb{R}^{m}_{\phi})\setminus E_{p} = \mathbb{R}^{m}_{0}\setminus B_{1}(p_{0})\cup\mathbb{R}^{m}_{\phi}\setminus B_{1}(p_{\phi})$.
Now define
\begin{align*}
U_{\widetilde{F}_{t}} = \left\{\begin{array}{ll}
                          U_{0},&\quad x\in\mathbb{R}^{m}_{0}\setminus B_{1}(p_{0}),\\
                          U_{\phi}^{\lambda},&\quad x\in\mathbb{R}^{m}_{\phi}\setminus B_{1}(p_{\phi}),\\
                          \tau_{dw_{t}}^{-1}(U_{N}),&\quad x\in\sigma\times [t\hat{R}, \epsilon],\\
                          t\cdot U_{N},&\quad x\in(S^{m-1}\times\mathbb{R})\setminus \hat{E}_{\infty},
                         \end{array}\right.
\end{align*}
\begin{align*}
\Phi_{\widetilde{F}_{t}} = \left\{\begin{array}{ll}
                          \Phi_{0},&\quad x\in\mathbb{R}^{m}_{0}\setminus B_{1}(p_{0}),\\
                          \Phi_{\phi}^{\lambda},&\quad x\in\mathbb{R}^{m}_{\phi}\setminus B_{1}(p_{\phi}),\\
                          \Phi_{N}\circ dw_{t},&\quad x\in\sigma\times [t\hat{R}, \epsilon],\\
                          t\Phi_{N}t^{-1},&\quad x\in(S^{m-1}\times\mathbb{R})\setminus \hat{E}_{\infty},
                         \end{array}\right.,
\end{align*}
where $\tau_{dw_{t}}$ is the symplectomorphism given by
\begin{align*}
    \tau_{dw_{t}}:T^{*}\mathbb{R}^{m}\to\ T^{*}\mathbb{R}^{m},\quad \tau_{dw_{t}}(x, y) := (x, y+dw_{t}(x)).
\end{align*}
Then $(U_{\widetilde{F}_{t}}, \Phi_{\widetilde{F}_{t}})$ is the desired Lagrangian neighbourhood for $\widetilde{F}_{t}:X_{t}\to\mathbb{C}^{m}$.

\subsection{Setting up the perturbation problem}
Let $(U_{\widetilde{F}_{t}}, \Phi_{\widetilde{F}_{t}})$ be the Lagrangian meighbourhood constructed in the last subsection.  Consider a closed $1$-form $\eta_{t}\in \Gamma(T^{*}X_{t})$ satisfies $\eta_{t}(X_{t})\in U_{\widetilde{F}_{t}}$, where we regard $\eta_{t}$ as a smooth map $\eta_{t}:X_{t}\to T^{*}X_{t}$. Let $\Gamma_{\eta_{t}} := \eta_{t}(X_{t})\subset U_{\widetilde{F}_{t}}$ denote the graph of $\eta_{t}$, then $\Phi_{\widetilde{F}_{t}}(\Gamma_{\eta_{t}})$ is a Lagrangian submanifold in $\mathbb{C}^{m}$ diffeomorphic to $X_{t}$. Recall that a Lagrangian submanifold $F:X\to\mathbb{C}^{m}$ is a translating soliton if and only if $\Theta(F) := *F^{*}\im\Omega_{f} = 0$. Thus for each small $t$, we want to solve for $\eta_{t}\in \Gamma(T^{*}X_{t})$ satisfying $\Theta_{t}:=*F_{\eta_{t}}^{*}\im\Omega_{f} = 0$, where $F_{\eta_{t}} := \Phi_{\widetilde{F}_{t}}\circ \eta_{t}: X_{t}\to\mathbb{C}^{m}$ is the Lagrangian immersion of the graph of $\eta_{t}$. To this end, we regard $\Theta_{t}$ as a mapping
\begin{align*}
\Theta_{t}: \mathcal{U}_{t}:=\{ \eta_{t}\in\Gamma(T^{*}X_{t})\:|\:\Gamma_{\eta_{t}}\subset U_{\widetilde{F}_{t}}\}\to C^{\infty}(X_{t}),\quad\Theta_{t}(\eta_{t}):=*F_{\eta_{t}}^{*}\im\Omega_{f}.
\end{align*}
Since the value of $\Theta_{t}$ also depends on the first derivative of $\eta_{t}$, we may think of $\Theta_{t}$ as being obtained from an underlying function
\begin{align*}
\Theta'_{t}: U_{\tilde{F}_{t}}\times\otimes^{2}T^{*}_{x}X_{t}\longrightarrow\mathbb{R}
\end{align*}  
satisfying $\Theta'_{t}[x, (\eta_{t})_{x}, (\nabla\eta_{t})_{x}] = \Theta_{t}(\eta_{t})(x)$. Note that while $\Theta_{t}$ is a mapping between infinite dimensional spaces, $\Theta'_{t}$ is a mapping between finite dimensional spaces.

The function $\Theta'_{t}$ can be defined as follows. Choose $(x, \eta)\in U_{\widetilde{F}_{t}}$, let $\{e_{j}\}_{j=1}^{m}$ be a positive orthonormal basis of $T_{x}X_{t}$. For any $\xi\in\otimes^{2}T^{*}_{x}X_{t}$ and $j = 1, \cdots, m$, denote $\iota_{e_{j}}\xi := \xi(e_{j}, \cdot)\in T^{*}_{x}X_{t}$. Then using the splitting $T_{(x, \eta)}U_{\widetilde{F}_{t}} = T_{x}X_{t}\oplus T^{*}_{x}X_{t}$, the vectors $(e_{1}, \iota_{e_{1}}\xi), \cdots, (e_{m}, \iota_{e_{m}}\xi)$ span an $m$-plane in $T_{(x, \eta)}X_{t}$. Define
\begin{align}\label{theta'}
\Theta'_{t}[x, \eta, \xi] := \Phi^{*}_{\tilde{F}_{t}}\im\Omega_{f}\big|_{(x, \eta)}(\:(e_{1}, \iota_{e_{1}}\xi), \cdots, (e_{m}, \iota_{e_{m}}\xi)\:).
\end{align}
It is straightforward to see that $\Theta'_{t}$ defined as above satisfies $\Theta'_{t}[x, (\eta_{t})_{x}, (\nabla\eta_{t})_{x}] = \Theta_{t}(\eta_{t})(x)$.

Define $\widetilde{\Theta}_{t}:=\Theta_{t}\circ d$. Then
by a direct computation, $\widetilde{\Theta}_{t}$ extends to a smooth map
\begin{align*}
\widetilde{\Theta}_{t}: \widetilde{\mathcal{U}}_{t} := \{\:u_{t}\in W^{k, p}_{\beta, \gamma, t}(X_{t})\:|\:\Gamma_{du_{t}}\subset U_{\widetilde{F}_{t}}\}\to W^{k-2, p}_{\beta+1, \gamma-2, t}(X_{t})
\end{align*}
for $k\geq 2$, $p>m$, $\beta, \gamma\in\mathbb{R}$. Note that by Sobolev embedding (Theorem $\ref{sobolevemb}$), $p>m$ implies that $u_{t}$ is in fact continuously differentiable, so the condition $\Gamma_{du_{t}}\subset U_{\widetilde{F}_{t}}$ makes sense. By \cite[Proposition 5.6]{Joyce3}, we may write
\begin{align*}
\widetilde{\Theta}_{t}(u_{t}) = \widetilde{\Theta}_{t}(0) + d\widetilde{\Theta}_{t}\big|_{0}(u_{t}) + Q_{t}(du_{t}),
\end{align*}
where $d\widetilde{\Theta}_{t}\big|_{0}(u_{t}) = \mathcal{L}_{g_{t}}u_{t}$, and $|Q_{t}(du_{t})| = O(|du_{t}|^{2} + |\nabla du_{t}|^{2})$ for small $du_{t}$.

We have estimated $\widetilde{\Theta}_{t}(0)$ in \S4.2 and show the invertibility of $d\widetilde{\Theta}_{t}\big|_{0}$ in \S5.5. It remains to estimate the quadratic term $Q_{t}$.

\subsection{Estimate of the quadratic term}
To estimate $Q_{t}$, we should first estimate the  derivatives of $\Theta'_{t}$. Fix $x\in X_{t}$, let $\partial_{1}, \partial_{2}$ denote the partial derivatives in $\eta, \xi$ direction, respectively. Then we would like to estimate the partial derivatives $\partial_{i}\partial_{j}\Theta'_{t}$ and $\partial_{i}\partial_{j}\partial_{k}\Theta'_{t}$, $i, j, k = 1, 2$.

To this end, we recall the connection on $TU_{\widetilde{F}_{t}}$ defined by Joyce \cite[Definition 5.2]{Joyce3}. The Levi-Civita connection $\nabla$ of the induced metric $g_{t}$ on $X_{t}$ induces a splitting $TU_{\widetilde{F}_{t}} = \mathcal{H}\oplus\mathcal{V}$ into horizontal subbundle $\mathcal{H}\simeq TX_{t}$ and vertical subbundle $\mathcal{V}\simeq T^{*}X_{t}$. Define a connection $\hat{\nabla}$ on $TU_{\tilde{F}_{t}}$ by lifting the Levi-Civita connection $\nabla$ on $\mathcal{H}$, and by using partial differentiation on $\mathcal{V}$. Then it follows from $(\ref{theta'})$ that the partial derivatives of  $\Theta'_{t}$ can be estimated by the covariant derivatives of $\Phi^{*}_{\widetilde{F}_{t}}\im\Omega_{f}$ using $\hat{\nabla}$. For instance, we have
\begin{align*}
|\partial_{1}\partial_{1}\Theta'_{t}|\leq C_{11}|{\hat{\nabla}}^{2}(\Phi^{*}_{\widetilde{F}_{t}}\im\Omega_{f})|, \quad |\partial_{1}\partial_{2}\Theta'_{t}|\leq C_{12}|{\hat{\nabla}}(\Phi^{*}_{\widetilde{F}_{t}}\im\Omega_{f})|, \quad |\partial_{2}\partial_{2}\Theta'_{t}|\leq C_{22}|\Phi^{*}_{\widetilde{F}_{t}}\im\Omega_{f}|
\end{align*}
for some $C_{11}, C_{12}, C_{22}>0$. Now, from the the estimates by Pacini \cite[p250]{PaciniGlue} and our construction of Lagrangian neighborhood, for any $(x, \eta)\in U_{\widetilde{F}_{t}}$, there exist constants $D_{k}>0$ such that
\begin{align}\label{estconn}
    |\hat{\nabla}^{k}\Phi^{*}_{\tilde{F}_{t}}\im\Omega_{f}(x, \eta)|\leq D_{k}e^{-\frac{f_{t}(x)}{2}}\rho^{-k}_{t}(x).
\end{align}
Therefore, there exists $C>0$ such that
\begin{align}\label{esttheta'}
|\partial_{1}\partial_{1}\Theta'_{t}|\leq Ce^{-\frac{f_{t}}{2}}\rho_{t}^{-2},\quad |\partial_{1}\partial_{2}\Theta'_{t}|\leq Ce^{-\frac{f_{t}}{2}}\rho_{t}^{-1},\quad |\partial_{2}\partial_{2}\Theta'_{t}|\leq Ce^{-\frac{f_{t}}{2}}.
\end{align}
Note that the power of $\rho_{t}^{-1}$ depends only on how many times that we apply $\partial_{1}$, since the derivatives in $\xi$-direction does not involve the covariant derivatives using $\hat{\nabla}$. 

\

\noindent{\it $C^{1}$-estimate of $Q_{t}$.}\quad
We prove the $C^{1}$-estimate for $Q_{t}$, following \cite[Proposition 5.8]{Joyce3} and  \cite[Remark 5.4]{PaciniGlue}. First we do the zeroth order estimate.
\begin{lemma}\label{C0}
Fix $x\in X_{t}$. For any $\alpha, \beta\in \mathcal{U}_{t}$, there exists $C>0$ such that
\begin{align}
    \left|\:[Q_{t}(\alpha) - Q_{t}(\beta)](x)\:\right|\leq Ce^{-\frac{f_{t}}{2}}(\rho_{t}^{-1}|\alpha - \beta| + |\nabla\alpha - \nabla\beta|)(\rho_{t}^{-1}|\alpha| + \rho_{t}^{-1}|\beta| + |\nabla\alpha|+|\nabla\beta|),
\end{align}
where the right hand side is evaluated at $x$.
\end{lemma}
\begin{proof}
Since $\mathcal{U}_{t}$ is convex, the curve $\gamma(s):=s\alpha + (1-s)\beta$ for $s\in [0, 1]$ is contained in $\mathcal{U}_{t}$. Then
\begin{align*}
    Q_{t}(\alpha) - Q_{t}(\beta) = Q_{t}(\gamma(1)) - Q_{t}(\gamma(0)) = \int_{0}^{1}\frac{d}{ds}Q_{t}(\gamma(s))\:ds.
\end{align*}
Since $\Theta_{t}(\gamma(s))= \Theta_{t}(0) + d\Theta_{t}\big|_{0}(\gamma(s)) + Q_{t}(\gamma(s))$ and $\gamma'(s) = \alpha - \beta$,
\begin{align*}
    \frac{d}{ds}Q_{t}(\gamma(s)) &= \frac{d}{ds}\Theta_{t}(\gamma(s)) - d\Theta_{t}\big|_{0}(\gamma'(s)) = d\Theta_{t}\big|_{\gamma(s)}(\gamma'(s)) - d\Theta_{t}\big|_{0}(\gamma'(s))\\
    &=\int_{0}^{s}\frac{d}{dl}d\Theta_{t}\big|_{\gamma(l)}(\alpha - \beta)\:dl = \int_{0}^{s}(\nabla d\Theta_{t})_{\gamma(l)}(\alpha - \beta, \alpha - \beta)\:dl,
\end{align*}
where $\nabla d\Theta_{t}$ is the Hessian of $\Theta_{t}$ on $\mathcal{U}_{t}$. Since $\Theta'[x, \alpha, \nabla\alpha] = \Theta_{t}(\alpha)(x)$, we can replace $\nabla d\Theta_{t}$ by the Hessian of $\Theta'_{t}$, which can be expressed using partial derivatives $\partial_{i}\partial_{j}\Theta'_{t}$.
Indeed, we have
\begin{align*}
    (\nabla d\Theta_{t})_{\gamma(l)}(\alpha - \beta, \alpha - \beta) = (\hat{\nabla}d\Theta'_{t})_{(\gamma(l), \nabla\gamma(l))}[(\alpha-\beta, \nabla(\alpha - \beta)), (\alpha-\beta, \nabla(\alpha - \beta))].
\end{align*}
Therefore, by $(\ref{esttheta'})$, 
\begin{align*}
    \left|\frac{d}{ds}Q_{t}(\gamma(s))\right|&\leq Ce^{-\frac{f_{t}}{2}}(\:\rho_{t}^{-2}|\alpha - \beta|^{2} + 2\rho_{t}^{-1}|\alpha - \beta||\nabla\alpha - \nabla\beta| + |\nabla\alpha - \nabla\beta|^{2}\:)\\
    &= Ce^{-\frac{f_{t}}{2}}(\rho_{t}^{-1}|\alpha - \beta| + |\nabla\alpha - \nabla\beta|)(\rho_{t}^{-1}|\alpha - \beta| + |\nabla\alpha - \nabla\beta||)\\
    &\leq Ce^{-\frac{f_{t}}{2}}(\rho_{t}^{-1}|\alpha - \beta| + |\nabla\alpha - \nabla\beta|)(\rho_{t}^{-1}(|\alpha| + |\beta|) + (|\nabla\alpha|+|\nabla\beta|)),
\end{align*}
and the result follows.
\end{proof}

The first order estimate follows similarly.
\begin{lemma}\label{C1}
Fix $x\in X_{t}$ and given $\alpha, \beta\in \mathcal{U}_{t}$. Regard $Q_{t}(\alpha)$ and $Q_{t}(\beta)$ as functions on $X_{t}$. Then there exists $C>0$ such that
\begin{align*}
    |\:[\rho_{t}\:d(Q_{t}(\alpha) - Q_{t}(\beta))](x)\:&|\leq Ce^{-\frac{f_{t}}{2}}\{\:\rho_{t}^{-1}|\alpha - \beta|(\rho_{t}^{-1}|\alpha| + \rho_{t}^{-1}|\beta|) + \rho_{t}^{-1}|\alpha - \beta|(|\nabla\alpha| + |\nabla\beta|)\\
    &+\rho_{t}^{-1}|\alpha - \beta|(\rho_{t}|\nabla^{2}\alpha| + \rho_{t}|\nabla^{2}\beta|) + |\nabla\alpha - \nabla\beta|(\rho_{t}^{-1}|\alpha| + \rho_{t}^{-1}|\beta|)\\
    &+ |\nabla\alpha - \nabla\beta|(|\nabla\alpha| + |\nabla\beta|) + |\nabla\alpha - \nabla\beta|(\rho_{t}|\nabla^{2}\alpha| + \rho_{t}|\nabla^{2}\beta|)\\
    &+\rho_{t}|\nabla^{2}\alpha - \nabla^{2}\beta|(\rho_{t}^{-1}|\alpha| + \rho_{t}^{-1}|\beta|) + \rho_{t}|\nabla^{2}\alpha - \nabla^{2}\beta|(|\nabla\alpha| + |\nabla\beta|)\:\}.
\end{align*}
\end{lemma}
\begin{proof}
Let $\gamma(s) = s\alpha + (1-s)\beta$, $s\in [0, 1]$. As in the proof of Lemma $\ref{C0}$, 
\begin{align*}
    Q_{t}(\alpha) - Q_{t}(\beta) =  \int_{0}^{1}\frac{d}{ds}Q_{t}(\gamma(s))\:ds=\int_{0}^{1}\int_{0}^{s}(\nabla d\Theta_{t})_{\gamma(l)}(\alpha - \beta, \alpha - \beta)\:dl\:ds.
\end{align*}
Hence
\begin{align*}
    d(Q_{t}(\alpha) - Q_{t}(\beta) ) = \int_{0}^{1}\int_{0}^{s}d[(\nabla d\Theta_{t})_{\gamma(l)}(\alpha - \beta, \alpha - \beta)]\:dl\:ds.
\end{align*}
Express the Hessian $\nabla d\Theta_{t}$ by using partial derivatives $\partial_{i}\partial_{j}\Theta'_{t}$ and use $(\ref{estconn})$, we have
\begin{align*}
    |d[(\nabla d\Theta_{t})_{\gamma(l)}(\alpha - \beta, \alpha - &\beta)]|\\
    \leq Ce^{-\frac{f_{t}}{2}}\{\:&\rho_{t}^{-3}|\alpha - \beta|^{2} + 4\rho_{t}^{-2}|\alpha - \beta||\nabla\alpha - \nabla\beta| + 3\rho_{t}^{-1}|\nabla\alpha - \nabla\beta|^{2}\\
    &+ 2\rho_{t}^{-1}|\alpha - \beta||\nabla^{2}\alpha - \nabla^{2}\beta| + 2|\nabla\alpha - \nabla\beta||\nabla^{2}\alpha - \nabla^{2}\beta|\:\}
\end{align*}
for some $C>0$. Substitute some $|\nabla^{k}\alpha - \nabla^{k}\beta|$ by $|\nabla^{k}\alpha| + |\nabla^{k}\beta|$, $k = 0, 1, 2$, and manipulate, we obtain the result.
\end{proof}

Now we assume $\alpha = du$, $\beta = dv$ for $u, v\in C^{\infty}(X_{t})$.
\begin{corollary}
There exists $C>0$ such that
\begin{align}\label{zero}
    \left|\:[Q_{t}(du) - Q_{t}(dv)](x)\:\right|\leq Ce^{-\frac{f_{t}}{2}}(\rho_{t}^{-1}|du - dv| + |\nabla du - \nabla dv|)(\:\|u\|_{C^{2}_{0, 2, t}} + \|v\|_{C^{2}_{0, 2, t}}\:).
\end{align}
and
\begin{align}\label{1}
    |\:[\:d(Q_{t}(du) - Q_{t}(dv))](x)\:|\leq  Ce^{-\frac{f_{t}}{2}}&\{\:(\rho_{t}^{-2}|du - dv| + \rho_{t}^{-1}|\nabla du - \nabla dv|\nonumber\\
    &\qquad+ |\nabla^{2}du - \nabla^{2}dv|)(\:\|u\|_{C^{2}_{0, 2, t}} + \|v\|_{C^{2}_{0, 2, t}})\nonumber\\
    &\qquad+\|u - v\|_{C^{2}_{0, 2, t}}(|\nabla^{2}du| + |\nabla^{2}dv|)\:\}.
\end{align}
\end{corollary}

Next we prove the integral estimate of $Q_{t}$.
\begin{proposition}\label{estQ}
For any $u, v\in C^{\infty}(X_{t})$ with $du, dv\in U_{\tilde{F}_{t}}$, 
\begin{align*}
\|Q_{t}(du) - &Q_{t}(dv)\|_{W^{1, p}_{\beta+1, \gamma-2, t}}\\
&\leq C\left\{(\:\|u\|_{C^{2}_{0, 2, t}} + \|v\|_{C^{2}_{0, 2, t}}\:)\|u - v\|_{W^{3, p}_{\beta, \gamma, t}} + (\:\|u\|_{W^{3, p}_{\beta, \gamma, t}}  + \|v\|_{W^{3, p}_{\beta, \gamma, t}} ) \|u - v\|_{C^{2}_{0, 2, t}}\right\}.
\end{align*}
\end{proposition}
\begin{proof}
By $(\ref{zero})$, we estimate
\begin{align}\label{1'}
   \int_{X_{t}}|e^{\frac{(\beta+1)f_{t}}{2}}&\rho_{t}^{-\gamma+2}(Q_{t}(du) - Q_{t}(dv))|^{p}\rho_{t}^{-m}\:dV_{g_{t}}\nonumber\\
   &\leq C(\:\|u\|_{C^{2}_{0, 2, t}} + \|v\|_{C^{2}_{0, 2, t}}\:)^{p}\int_{X_{t}}|e^{\frac{\beta f_{t}}{2}}\rho_{t}^{-\gamma+2}(\rho_{t}^{-1}|du - dv| + |\nabla du - \nabla dv|)|^{p}\rho_{t}^{-m}\:dV_{g_{t}}\nonumber\\
   &\leq C(\:\|u\|_{C^{2}_{0, 2, t}} + \|v\|_{C^{2}_{0, 2, t}}\:)^{p}\:\|u - v\|_{W^{2, p}_{\beta, \gamma, t}}^{p}.
\end{align}
By $(\ref{1})$,
\begin{align*}
    \int_{X_{t}}|e^{\frac{(\beta + 1)f_{t}}{2}}&\rho_{t}^{-\gamma+3}d[(Q_{t}(du) - Q_{t}(dv))]|^{p}\rho_{t}^{-m}\:dV_{g_{t}}\\
    &\leq C(\:\|u\|_{C^{2}_{0, 2, t}} + \|v\|_{C^{2}_{0, 2, t}})^{p}\sum_{j = 1}^{3}\int_{X_{t}}|e^{\frac{\beta f_{t}}{2}}\rho_{t}^{-\gamma + j}\nabla^{j}(u - v)|^{p}\rho_{t}^{-m}\:dV_{g_{t}}\\
    &\qquad\qquad+ C\|u - v\|_{C^{2}_{0, 2, t}}^{p}\int_{X_{t}}|e^{\frac{\beta f_{t}}{2}}\rho_{t}^{-\gamma+3}(|\nabla^{2}du| + |\nabla^{2}dv|)|^{p}\rho_{t}^{-m}\:dV_{g_{t}}.
\end{align*}
Hence we find that there exists $C>0$ such that
\begin{align}\label{3'}
    \int_{X_{t}}|e^{\frac{(\beta + 1)f_{t}}{2}}&\rho_{t}^{-\gamma+3}d[(Q_{t}(du) - Q_{t}(dv))]|^{p}\rho_{t}^{-m}\:dV_{g_{t}}\nonumber\\
    &\leq C(\:\|u\|_{C^{2}_{0, 2, t}} + \|v\|_{C^{2}_{0, 2, t}})^{p}\|u-v\|_{W^{3, p}_{\beta, \gamma, t}}^{p} + \|u - v\|_{C^{2}_{0, 2, t}}^{p}(\:\|u\|_{W^{3, p}_{\beta, \gamma, t}}^{p} + \|v\|_{W^{3, p}_{\beta, \gamma, t}}^{p}).
\end{align}
Combining $(\ref{1'})$ and $(\ref{3'})$, we obtain the result.
\end{proof}

For $T$-finite approximate solutions $\widetilde{F}_{t}:X_{t}\to\mathbb{C}^{m}$, we are able to obtain the complete quadratic estimate, due to the Sobolev embedding (Theorem $\ref{sobolevemb}$) $W^{3, p}_{\beta, \gamma, t}\hookrightarrow C^{2}_{0, 2, t}$ for $\beta<0$ and $\gamma<2$. 
\begin{corollary}\label{QuadraticEst}
Suppose the approximate solution $\widetilde{F}_{t}:X_{t}\to\mathbb{C}^{m}$ is $T$-finite, then for $p>m$, $\beta\in (-1, 0)$, $\gamma\in (2-m, 0)$ and $t\in (0, \delta)$, there exists $C>0$ independent of $t$ such that 
\begin{align}
    \|Q_{t}(du) - Q_{t}(dv)\|_{W^{1, p}_{\beta+1, \gamma-2, t}}\leq Ct^{\gamma - 2}\|u - v\|_{W^{3, p}_{\beta, \gamma, t}}(\:\|u\|_{W^{3, p}_{\beta, \gamma, t}} + \|v\|_{W^{3, p}_{\beta, \gamma, t}}\:).
\end{align}
\end{corollary}

\subsection{Perturbation of $T$-finite approximate solutions}

We are ready to solve the nonlinear equation
\begin{align*}
\widetilde{\Theta}_{t}(0) + \mathcal{L}_{g_{t}}u_{t} + Q_{t}(u_{t}) = 0
\end{align*}
on $T$-finite approximate solutions using the Fix Point Theorem.

Let $p>m$, $\beta\in (-1, 0)$ and $\gamma\in (2-m, 0)$. By Proposition $\ref{LinearIsom}$, there exists  a small $\delta>0$ such that for any $t\in (0, \delta)$, there exists $C>0$ independent of $t$ so that for any $u_{t}\in W^{3, p}_{\beta, \gamma, t}(X_{t})$, 
\begin{align*}
\|u_{t}\|_{W^{3, p}_{\beta, \gamma, t}}\leq C\:\|\mathcal{L}_{g_{t}} u_{t}\|_{W^{1,p}_{\beta+1, \gamma-2, t}}.
\end{align*}
Set 
\begin{align*}
\mathcal{G}_{t}: W^{3, p}_{\beta, \gamma, t}(X_{t})\to W^{3, p}_{\beta, \gamma, t}(X_{t}),\quad
\mathcal{G}_{t}(u_{t}):=-\mathcal{L}^{-1}_{g_{t}}(\widetilde{\Theta}_{t}(0) + Q_{t}(du_{t})).
\end{align*}
Then it follows from the above uniform invertibility of $\mathcal{L}_{g_{t}}$ that for $t\in (0, \delta)$, there exists $C>0$ such that
\begin{align}\label{EstG}
\|\mathcal{G}_{t}(u_{t})\|_{W^{3, p}_{\beta, \gamma, t}}\leq C\left(\|\widetilde{\Theta}_{t}(0)\|_{W^{1, p}_{\beta+1, \gamma-2, t}} + \|Q_{t}(du_{t})\|_{W^{1, p}_{\beta+1, \gamma-2, t}}\right),
\end{align}
and for $u_{t}, v_{t}\in W^{3, p}_{\beta, \gamma, t}(X_{t})$,
\begin{align}\label{EstGG}
    \|\mathcal{G}_{t}(u_{t}) - \mathcal{G}_{t}(v_{t})\|_{W^{3, p}_{\beta, \gamma, t}}\leq C \|Q_{t}(du_{t}) - Q_{t}(dv_{t})\|_{W^{1, p}_{\beta+1, \gamma-2, t}}.
\end{align}

Let $B_{t^{\alpha}}:=\{\:u\in W^{3, p}_{\beta, \gamma, t}(X_{t})\::\:\|u\|_{W^{3, p}_{\beta, \gamma, t}}<t^{\alpha}\:\}$. Then $B_{t^{\alpha}}$ is clearly convex. Note that for $\tau<1$ and $\gamma>2-m$,
\begin{align*}
    \tau(2-\gamma) + (1-\tau)m - (2-\gamma) = (1-\tau)(\gamma-(2-m))>0.
\end{align*}
\begin{lemma}\label{Aut}
Let $\tau\in (\frac{m}{m+1}, 1)$. Choose $\alpha$ such that $2-\gamma<\alpha<\tau(2-\gamma) + (1-\tau)m$.  Then there exists $\delta'\in (0, \delta]$ so that for all $t\in (0, \delta')$, we have $\mathcal{G}_{t}(B_{t^{\alpha}})\subseteq B_{t^{\alpha}}$.
\end{lemma}
\begin{proof}
By Proposition $\ref{InitialEst}$ and Corollary $\ref{QuadraticEst}$, for $u\in B_{t^{\alpha}}$, there exists $C>0$ such that
\begin{align*}
\|\mathcal{G}_{t}(u)\|_{W^{3, p}_{\beta, \gamma, t}}\leq C\left( t^{\tau(2-\gamma) + (1-\tau)m - \alpha}+ t^{\alpha- (2-\gamma)}\right)\cdot  t^{\alpha}.
\end{align*}
By the choice of $\alpha$, we may choose $\delta'>0$ small enough so that $t^{\tau(2-\gamma) + (1-\tau)m - \alpha}+ t^{\alpha - (2-\gamma)}<\frac{1}{2C}$ for $t\in (0, \delta')$, then the result follows.
\end{proof}
\begin{lemma}\label{Contraction}
For the same choice of $\alpha$ as in Lemma $\ref{Aut}$, there exists $\delta''\in(0, \delta]$ such that for all $t\in (0, \delta'')$, $\mathcal{G}_{t}: B_{t^{\alpha}}\to B_{t^{\alpha}}$ is a contraction mapping.
\end{lemma}
\begin{proof}
By estimate ($\ref{EstGG}$) and Corollary $\ref{QuadraticEst}$, there exists $C'>0$ such that for any $u, v\in B_{t^{\alpha}}$,
\begin{align*}
\|\mathcal{G}_{t}(u) - \mathcal{G}_{t}(v)\|_{W^{3, p}_{\beta, \gamma, t}}&\leq C'\|Q_{t}(du) - Q_{t}(dv)\|_{W^{1, p}_{\beta+1, \gamma-2, t}}\\
&\leq C'\cdot t^{\alpha - (2-\gamma)}\|u - v\|_{W^{3, p}_{\beta, \gamma, t}}.
\end{align*} 
It remains to choose $\delta''>0$ small enough so that for $t\in (0, \delta'')$, $t^{\alpha - (2-\gamma)}<\frac{1}{2C'}$.
\end{proof}
By Lemma $\ref{Aut}$, Lemma $\ref{Contraction}$, we may apply the Fixed Point Theorem to conclude the following.
\begin{proposition}\label{solution}
Let $m\geq 3$, $p>m$, $\beta\in (-1, 0)$, $\gamma\in(2-m, 0)$. Choose $\alpha$ as in Lemma $\ref{Aut}$. Then for $t\in (0, \min\{\delta', \delta''\})$, there exists a unique fix point $u_{t}\in B_{t^{\alpha}}$ of $\mathcal{G}_{t}$, and this $u_{t}\in W^{3, p}_{\beta, \gamma, t}(X_{t})$ is the unique solution for $\widetilde{\Theta}_{t}(u_{t}) = 0$.
\end{proposition}

We can finally prove the existence of smooth Lagrangian translating soliton near a $T$-finite approximate solution.
\begin{theorem}
Let $m\geq 3$, $T = (0, \cdots, 0, -1)$, and let $\widetilde{F}:L = \mathbb{R}_{0}^{m}\cup\mathbb{R}_{\phi}^{m}\to\mathbb{C}^{m}$ be the $T$-finite Lagrangian translating soliton with an isolated conical singularity defined in \S 3.4. Let $\widetilde{F}_{t}: X_{t}\to\mathbb{C}^{m}$ be a family of approximate solutions which is constructed by gluing Lawlor necks to the isolated conical singularity of $\widetilde{F}(L)$, as demonstrated in \S4.1. Then there exists $\delta>0$ and a $1$-parameter family of Lagrangian translating solitons $\{F_{t}:X_{t}\to\mathbb{C}^{m}\}_{t\in (0, \delta)}$ satisfying $H_{F_{t}} + T^{\perp} = 0$ and  $F_{t}(X_{t})\to \widetilde{F}(L)$ as $t\to 0$ as currents.
\end{theorem}
\begin{proof}
By Proposition $\ref{solution}$, we can solve $\widetilde{\Theta}_{t}(u_{t}) = 0$ in $B_{t^{\alpha}}\subset W^{3, p}_{\beta, \gamma, t}(X_{t})$ for $p>m$, $\beta\in (-1, 0)$ and $\gamma\in (2-m, 0)$. By Sobolev embedding (Theorem $\ref{sobolevemb}$), the solution $u_{t}$ is in $C^{2}_{\beta, \gamma, t}(X_{t})$. Now $u_{t}$ is a $C^{2}$ solution of a nonlinear second order equation $\Theta'_{t}[x, du_{t}(x), \nabla du_{t}(x)] = 0$ (see \S 6.2) where $\Theta'_{t}$ is smooth in its arguments, so we may apply the regularity result \cite[Theorem 3.56]{Aubin} to conclude that $u_{t}\in C^{\infty}(X_{t})$. Hence $F_{t}:=\Phi_{\widetilde{F}_{t}}\circ du_{t}:X_{t}\to\mathbb{C}^{m}$ defines a $1$-parameter family of smooth Lagrangian translating soliton with phase $0$, which is $t^{\alpha}$-close to $\widetilde{F}_{t}$ in $C^{1}$ sense for some $\alpha>0$ as chosen in Lemma $\ref{Aut}$. By construction, $\widetilde{F}_{t}(X_{t})\to \widetilde{F}(L)$ in current sense as $t\to 0$, so the same property holds for $F_{t}(X_{t})$. 
\end{proof}

\bibliographystyle{alpha}
\bibliography{LagTranslator.bib}

\begin{bibdiv}
\begin{biblist}

\bib{Aubin}{book}{
      author={Aubin, Thierry},
       title={Some nonlinear problems in {R}iemannian geometry},
      series={Springer Monographs in Mathematics},
   publisher={Springer-Verlag, Berlin},
        date={1998},
        ISBN={3-540-60752-8},
         url={https://doi.org/10.1007/978-3-662-13006-3},
      review={\MR{1636569}},
}

\bib{LM}{article}{
      author={B., Lockhart~Robert},
      author={McOwen, Robert~C.},
       title={Elliptic differential operators on noncompact manifolds},
        date={1985},
        ISSN={0391-173X},
     journal={Ann. Scuola Norm. Sup. Pisa Cl. Sci. (4)},
      volume={12},
      number={3},
       pages={409\ndash 447},
         url={http://www.numdam.org/item?id=ASNSP_1985_4_12_3_409_0},
      review={\MR{837256}},
}

\bib{biquard2017steady}{misc}{
      author={Biquard, Olivier},
      author={Macbeth, Heather},
       title={Steady k\"ahler-ricci solitons on crepant resolutions of finite
  quotients of $\mathbb{C}^n$},
        date={2017},
}

\bib{CL}{article}{
      author={Castro, Ildefonso},
      author={Lerma, Ana~M.},
       title={Translating solitons for {L}agrangian mean curvature flow in
  complex {E}uclidean plane},
        date={2012},
        ISSN={0129-167X},
     journal={Internat. J. Math.},
      volume={23},
      number={10},
       pages={1250101, 16},
         url={https://doi.org/10.1142/S0129167X12501017},
      review={\MR{2999046}},
}

\bib{conlon2020steady}{misc}{
      author={Conlon, Ronan~J.},
      author={Deruelle, Alix},
       title={Steady gradient k\"ahler-ricci solitons on crepant resolutions of
  calabi-yau cones},
        date={2020},
}

\bib{Transsurface2017}{article}{
      author={D\'{a}vila, Juan},
      author={del Pino, Manuel},
      author={Nguyen, Xuan~Hien},
       title={Finite topology self-translating surfaces for the mean curvature
  flow in {$\Bbb{R}^3$}},
        date={2017},
        ISSN={0001-8708},
     journal={Adv. Math.},
      volume={320},
       pages={674\ndash 729},
         url={https://doi.org/10.1016/j.aim.2017.09.014},
      review={\MR{3709119}},
}

\bib{HL}{article}{
      author={Harvey, Reese},
      author={Lawson, H.~Blaine, Jr.},
       title={Calibrated {G}eometries},
        date={1982},
        ISSN={0001-5962},
     journal={Acta Math.},
      volume={148},
       pages={47\ndash 157},
         url={https://doi.org/10.1007/BF02392726},
      review={\MR{666108}},
}

\bib{IJO}{article}{
      author={Imagi, Yohsuke},
      author={Joyce, Dominic},
      author={Oliveira~dos Santos, Joana},
       title={Uniqueness results for special {L}agrangians and {L}agrangian
  mean curvature flow expanders in {$\Bbb{C}^m$}},
        date={2016},
        ISSN={0012-7094},
     journal={Duke Math. J.},
      volume={165},
      number={5},
       pages={847\ndash 933},
         url={https://doi.org/10.1215/00127094-3167275},
      review={\MR{3482334}},
}

\bib{JoyceCS1}{article}{
      author={Joyce, Dominic},
       title={Special {L}agrangian submanifolds with isolated conical
  singularities. {I}. {R}egularity},
        date={2004},
        ISSN={0232-704X},
     journal={Ann. Global Anal. Geom.},
      volume={25},
      number={3},
       pages={201\ndash 251},
         url={https://doi.org/10.1023/B:AGAG.0000023229.72953.57},
      review={\MR{2053761}},
}

\bib{Joyce3}{article}{
      author={Joyce, Dominic},
       title={Special {L}agrangian submanifolds with isolated conical
  singularities. {III}. {D}esingularization, the unobstructed case},
        date={2004},
        ISSN={0232-704X},
     journal={Ann. Global Anal. Geom.},
      volume={26},
      number={1},
       pages={1\ndash 58},
         url={https://doi.org/10.1023/B:AGAG.0000023231.31950.cc},
      review={\MR{2054578}},
}

\bib{JoyceConj}{article}{
      author={Joyce, Dominic},
       title={Conjectures on {B}ridgeland stability for {F}ukaya categories of
  {C}alabi-{Y}au manifolds, special {L}agrangians, and {L}agrangian mean
  curvature flow},
        date={2015},
        ISSN={2308-2151},
     journal={EMS Surv. Math. Sci.},
      volume={2},
      number={1},
       pages={1\ndash 62},
         url={https://doi.org/10.4171/EMSS/8},
      review={\MR{3354954}},
}

\bib{JLT}{article}{
      author={Joyce, Dominic},
      author={Lee, Yng-Ing},
      author={Tsui, Mao-Pei},
       title={Self-similar solutions and translating solitons for {L}agrangian
  mean curvature flow},
        date={2010},
        ISSN={0022-040X},
     journal={J. Differential Geom.},
      volume={84},
      number={1},
       pages={127\ndash 161},
         url={http://projecteuclid.org/euclid.jdg/1271271795},
      review={\MR{2629511}},
}

\bib{Krylov}{book}{
      author={Krylov, N.V.},
       title={Lectures on {E}lliptic and {P}arabolic {E}quations in {S}obolev
  {S}paces},
      series={Graduate Studies in Mathematics, Graduate Studies in Mathema},
   publisher={American Mathematical Society},
        date={2008},
        ISBN={9780821846841},
         url={https://books.google.com.tw/books?id=8RzqBwAAQBAJ},
}

\bib{Lawlor}{article}{
      author={Lawlor, Gary},
       title={The angle criterion},
        date={1989},
        ISSN={0020-9910},
     journal={Invent. Math.},
      volume={95},
      number={2},
       pages={437\ndash 446},
         url={https://doi.org/10.1007/BF01393905},
      review={\MR{974911}},
}

\bib{Marshall}{thesis}{
      author={Marshall, Stephen},
       title={Deformations of special lagrangian submanifolds},
        type={Ph.D. Thesis},
        date={2002},
}

\bib{NevesSing}{article}{
      author={Neves, Andr\'e},
       title={Singularities of {L}agrangian mean curvature flow: zero-{M}aslov
  class case},
        date={2007},
        ISSN={0020-9910},
     journal={Invent. Math.},
      volume={168},
      number={3},
       pages={449\ndash 484},
         url={https://doi.org/10.1007/s00222-007-0036-3},
      review={\MR{2299559}},
}

\bib{Trident}{article}{
      author={Nguyen, Xuan~Hien},
       title={Translating tridents},
        date={2009},
     journal={Communications in Partial Differential Equations},
      volume={34},
      number={3},
       pages={257\ndash 280},
      eprint={https://doi.org/10.1080/03605300902768685},
         url={https://doi.org/10.1080/03605300902768685},
}

\bib{Transsurface2013}{article}{
      author={Nguyen, Xuan~Hien},
       title={Complete embedded self-translating surfaces under mean curvature
  flow},
        date={2013},
        ISSN={1050-6926},
     journal={J. Geom. Anal.},
      volume={23},
      number={3},
       pages={1379\ndash 1426},
         url={https://doi.org/10.1007/s12220-011-9292-y},
      review={\MR{3078359}},
}

\bib{Transsurface2015}{article}{
      author={Nguyen, Xuan~Hien},
       title={Doubly periodic self-translating surfaces for the mean curvature
  flow},
        date={2015},
        ISSN={0046-5755},
     journal={Geom. Dedicata},
      volume={174},
       pages={177\ndash 185},
         url={https://doi.org/10.1007/s10711-014-0011-2},
      review={\MR{3303047}},
}

\bib{PaciniUnifEst}{article}{
      author={Pacini, Tommaso},
       title={Desingularizing isolated conical singularities: {U}niform
  estimates via weighted {S}obolev spaces},
        date={201005},
     journal={Communications in Analysis and Geometry},
      volume={21},
}

\bib{PaciniGlue}{article}{
      author={Pacini, Tommaso},
       title={{S}pecial {L}agrangian conifolds, {II}: {G}luing constructions in
  $\mathbb{C}^m$},
        date={201109},
     journal={Proceedings of the London Mathematical Society},
      volume={107},
}

\bib{SmoWhitney}{article}{
      author={{Savas-Halilaj}, A.},
      author={{Smoczyk}, K.},
       title={{Lagrangian mean curvature flow of Whitney spheres}},
        date={2018-02},
     journal={ArXiv e-prints},
      eprint={1802.06304},
}

\bib{WBfmin}{article}{
      author={SU, Wei-Bo},
       title={{f-Minimal Lagrangian Submanifolds in Kähler Manifolds with Real
  Holomorphy Potentials}},
        date={201912},
        ISSN={1073-7928},
     journal={International Mathematics Research Notices},
  eprint={http://oup.prod.sis.lan/imrn/advance-article-pdf/doi/10.1093/imrn/rnz198/31224329/rnz198.pdf},
         url={https://doi.org/10.1093/imrn/rnz198},
        note={rnz198},
}

\bib{Weinstein}{article}{
      author={Weinstein, Alan},
       title={Symplectic manifolds and their {L}agrangian submanifolds},
        date={1971},
        ISSN={0001-8708},
     journal={Advances in Mathematics},
      volume={6},
      number={3},
       pages={329 \ndash  346},
  url={http://www.sciencedirect.com/science/article/pii/000187087190020X},
}

\end{biblist}
\end{bibdiv}

\end{document}